\documentclass[11pt,leqno]{amsart}
\usepackage{amssymb}
\usepackage{algorithm}
\usepackage{algorithmic}
\usepackage{amsmath}
\usepackage{amsfonts}
\usepackage{color,xcolor}
\usepackage{graphicx}
\usepackage{manfnt}
\usepackage{enumerate}
\usepackage{comment}
\usepackage{mathtools}
\usepackage{appendix}
\usepackage{epstopdf}
\usepackage{hyperref}

\newtheorem{Thm}{Theorem}

\newtheorem{prop}{Proposition}[section]
\newtheorem{remark}{Remark}[section]
\newtheorem{lemma}{Lemma}[section]

\newtheorem{definition}{Definition}[section]

\newtheorem{theo}[Thm]{Theorem}
 \numberwithin{equation}{section}
 \numberwithin{Thm}{section}

\newcommand{\be}{\begin{equation}}
\newcommand{\ee}{\end{equation}}
\newcommand\bes{\begin{eqnarray}}
\newcommand\ees{\end{eqnarray}}
\newcommand{\bess}{\begin{eqnarray*}}
\newcommand{\eess}{\end{eqnarray*}}

\title[Principal eigenvalue for a linear time-periodic parabolic operator]
{Asymptotics of the principal eigenvalue  for a linear time-periodic parabolic operator I: Large advection}
\author{Shuang Liu,\ \ Yuan Lou,\ \ Rui Peng\ \ and\ \, Maolin Zhou}

\thanks{{S. Liu}: Institute for Mathematical Sciences, Renmin University of China, Beijing 100872, P.R. China. Email: liushuangnqkg@ruc.edu.cn}

\thanks{{Y. Lou}: School of  Mathematical Sciences, Shanghai Jiao Tong University, Shanghai 200240, P.R. China and
Department of Mathematics, Ohio State University, Columbus, OH 43210, USA. Email:  lou@math.ohio-state.edu}

\thanks{{R. Peng}: School of Mathematics and Statistics, Jiangsu Normal University, Xuzhou, 221116, Jiangsu Province, P.R. China. Email:
pengrui\,$\b{}$\,seu@163.com}

\thanks{{M. Zhou}: Chern Institute of Mathematics, Nankai University, Tianjin 300071, P.R. China. Email:
zhouml123@nankai.edu.cn}

 
\begin{document}
\maketitle

\begin{abstract}
We investigate the  effect of large advection on the principal eigenvalues
of linear time-periodic parabolic operators with  zero Neumann boundary conditions. Various asymptotic behaviors of the principal eigenvalues, when advection coefficient approaches infinity,
are established,
where spatial or temporal degeneracy could occur in
the advection term.
Our findings substantially improve the results in \cite{PZ2015} for parabolic operators and also extend the existing results in \cite{CL2008,PZ2017} for elliptic operators.
\end{abstract}
\section{Introduction}\label{Introduction}\label{S1}
In this paper, we consider the following linear eigenvalue problem for
time-periodic parabolic operators  in one-dimensional space:
\begin{equation}\label{advection1}
 \left\{\begin{array}{ll}
 \medskip
\partial_{t}\varphi-D\partial_{xx}\varphi- \alpha \partial_x m
\partial_x\varphi+V\varphi=\lambda\varphi\ \ &{\text{in}}\,\,(0,1)\times[0,T],\\
 \medskip
 \partial_x\varphi(0,t)=\partial_x\varphi(1,t)=0 \ \ &{\text{on}}\,\,[0,T], \\
 \varphi(x,0)=\varphi(x,T) &\text{on}\,\,(0,1),
 \end{array}
 \right.
 \end{equation}
where the positive parameters
$D$ and $\alpha$ are
diffusion and advection rates, respectively. The functions
$m\in C^{2,0}([0,1]\times\mathbb{R})$ and $V\in C([0,1]\times\mathbb{R})$
are assumed to be periodic in $t$ with  a common period $T$.
It is well known \cite[Proposition 7.2]{Hess} that
problem \eqref{advection1} admits a principal eigenvalue $\lambda(\alpha)$, which is real and
simple, and the corresponding eigenfunction can be chosen to be positive. 
Furthermore, $\lambda(\alpha)<\mathrm{Re}(\lambda)$ holds  for any other eigenvalue $\lambda$ of \eqref{advection1}.


The goal of this paper is
to determine  the limit of $\lambda(\alpha)$ as $\alpha\to\infty$.
The asymptotic behavior of
the principal eigenvalue for \eqref{advection1}
as $D\to 0$ will be considered in 
\cite{LLPZ20202}.

\subsection{Background}
Besides its own mathematical interest,
the asymptotic behavior of the principal eigenvalue of
\eqref{advection1} for
large advection rate $\alpha$
is also motivated by its applications
to 
biological problems, including:
{\rm(i)} the phytoplankton growth with periodic 
light intensity; {\rm(ii)} the persistence and  competition of species in 
rivers with  spatio-temporally varying drift; 
{\rm(iii)}  the  competition of 
species  along  the gradient
of spatial-temporally varying resources; {\rm (iv)} the spreading of epidemic diseases in
time-periodic advective environments,
 among others. 
We refer to \cite{CC2003, C2014, LLL2020, LL2019_1} and references therein for further 
discussions. In the following
we highlight two
 potential applications of our results:

{\underline {Persistence for a single species}}.
The population dynamics of a single species, subject to  spatio-temporally  varying environmental
drifts, can be modeled as
 \begin{equation}\label{large_1_1}
\begin{cases}
 \begin{array}{ll}
 \smallskip
 \partial_t U=\partial_x\left(D\partial_x U-\alpha U \partial_x \tilde{m} \right)+Uf(x,t,U)
 \ \ & {\text{in}}\,\,(0,1)\times(0, \infty),\\
 D\partial_x U-\alpha U\partial_x \tilde{m} =0 & {\text{on}}\,\,\{0,1\}\times (0, \infty),
  \end{array}
  \end{cases}
 \end{equation}
where $\tilde{m}$ and $f$ are assumed to be $T$-periodic and $U$ denotes the density of the population; 
see \cite{CC2003,Hess,Ni2011} for more details. For instance,
model \eqref{large_1_1} can 
describe
 the phytoplankton  growth in light limited water columns  \cite{DH2008_1,DH2008_2,HOW1999, MO2017,PZ2015,PZ2016,S1981} or hydrobiological species  in rivers \cite{SG2001}.  An important biological  issue is
how  drift affects the population dynamics of \eqref{large_1_1} \cite{PLNL2005}.  Mathematically, the persistence of the single species in \eqref{large_1_1} is equivalent to the instability of trivial equilibrium $U\equiv0$ \cite{CC2003}, which is in turn determined by the sign of the
principal eigenvalue $\lambda(\alpha)$ of the linear problem \eqref{advection1} with  $V=-f(x,T-t,0)$
and $m=\tilde{m}(x, T-t)$,
by considering the corresponding adjoint operator.
If the environment  is spatio-temporally  varying, i.e. $\partial_x \tilde{m} $ and $f$
depend on $x$ and $t$ non-trivially
(e.g. phytoplankton population drifts up and down in the water column periodically in time,
and the light intensity at the water surface also varies periodically in time), then determining the sign of $\lambda(\alpha)$ becomes a challenging  issue, and standard spectral theory
is often not sufficient.

When  $m$ and $f$ are independent of time,
the high-dimensional version of \eqref{large_1_1} was initially proposed
in 
\cite{BC1995}, where
it is assumed that
the species may have the tendency to move upward along the resource gradient.
They investigated whether such directed movement could help promote the persistence of a single species. 
A related question is to determine the optimal distribution of resources for a single species to survive:  When $\alpha=0$,
it was shown in \cite{NY2018}  that, under a regularity assumption, the optimal distribution $m^*$ that maximizes the total biomass
of a single species is of the
``bang-bang" type, i.e. $m^*=\chi_{E}$ for some measurable set $E$,
where $\chi_{E}$ denotes the characteristic function of $E$; when  $\alpha\neq 0$, by studying the minimization of the principal eigenvalue of problem \eqref{advection1} with proper constraint on $m$,   the optimal distribution 
was fully determined in \cite{CDP2017}  for the one-dimensional case; see also \cite{MNP2019} for the recent progress on the higher dimensional case. We also refer to \cite{HNR2011} for  optimization problems on the principal eigenvalues for elliptic operators
with drift. It will be of interest to investigate the principal eigenvalue of problem \eqref{advection1}
under weaker regularity of $m$ and $V$, e.g.
when $m$ and/or $V$ are of the form $\chi_{E}$. Such questions are even not well understood for time-independent  $m$ and $V$.

{\underline{Competition for two species}}. The two competing species in time-periodic  and
advective environment can be modeled by the following
reaction-diffusion-advection system:
\begin{equation}\label{liu0}
\left\{
\begin{array}{ll}
\medskip
\partial_t U_1=\partial_x \left(D \partial_x U_1-\alpha_1 U_1 \partial_x m\right)+U_1 f(x, t, U_1, U_2) & \text{in } (0,1)\times (0,\infty),\\
\medskip
\partial_t U_2=\partial_x \left(D \partial_x U_2-\alpha_2 U_2 \partial_x m\right)+U_2 g(x, t, U_1, U_2) & \text{in } (0,1)\times (0,\infty),\\
D \partial_x U_1-\alpha_1 U_1 \partial_x m=D \partial_x U_2-\alpha_2 U_2 \partial_x m=0 & \text{on } \{0,1\}\times (0,\infty),
\end{array}
\right.
\end{equation}
 where $m$, $f$ and $g$ are $T$-periodic functions,
 and $U_1, U_2$ are the population densities of two species.
When the coefficients are time independent,  model \eqref{liu0} was first proposed in \cite{CCL2006} to study whether or not the advection along the resource gradient can confer competition advantage; see also \cite{ALL2017, CCL2007}.
When $\partial_x m$ is a positive constant,  \eqref{liu0} reduces to the river model studied in  \cite{LL2014_1, LL2014_2, LL2014_a, LP2015, LPZ2019,VL2011,ZZ2016}, where
the evolution of biased movement was considered. Since \eqref{liu0} generates a monotone dynamical system \cite{Hess},  the dynamics of \eqref{liu0}  is determined to a large extent by its steady periodic states and their stability properties. For example,
 the existence and stability of coexistence states for \eqref{liu0}
follow from the instability of the two semi-trivial states and
theory for  monotone systems \cite{Hess,H1988}. The stability of semi-trivial periodic state turns out to be determined by the sign of the principal eigenvalue of the linear problem \eqref{advection1} 
by considering the linearization
at this semi-trivial state.
Therefore,
the qualitative properties of principal eigenvalue with respect to $\alpha$  plays a pivotal role in analyzing the impact of the advection on the outcome of the competition
in time-periodic environments.

\subsection{Previous work}
Let $\lambda(\alpha)$ denote the principal eigenvalue of \eqref{advection1}. Observe that when $V$ depends on the time variable alone, i.e. $V(x,t)=V(t)$, there holds
$\lambda(\alpha)\equiv\frac{1}{T}\int_0^T V\left(s\right)\mathrm{d}s$ 
for all $\alpha$.
When $V$ and $\partial_xm$ depend on the space variable alone, i.e.
$V(x,t)=V(x)$ and $\partial_xm(x,t)=m'(x)$,  \eqref{advection1} reduces to
the following elliptic eigenvalue problem:
\begin{equation}\label{elliptic}
 \left\{\begin{array}{ll}
 \medskip
-D\varphi''- \alpha m'(x)\varphi'+V(x)\varphi=\lambda\varphi\ \ &\text{in}\,\,(0,1),\\
\varphi'(0)=\varphi'(1)=0.
 \end{array}
 \right.
 \end{equation}
For this simple-looking ODE eigenvalue problem, identifying the sign of $\lambda(\alpha)$
is not yet trivial. In the recent works \cite{CL2008,CL2012,LL2019,PZZ2018,PZ2017},
 the asymptotic behaviors of the principal eigenvalue for large $\alpha$ or small $D$  has been extensively studied for \eqref{elliptic} and its high dimensional version.

However, when $V$ or/and $\partial_xm$ depend both on the spatio-temporal variables,
much less has been known about the asymptotic behaviors of $\lambda(\alpha)$,
partly due to
the lack of variational structure for problem \eqref{advection1}. One may refer to
\cite{DT2016,DP2012,Hess,Hutson2000,LLPZ2019,Nadin2009,Nadin2011,PZ2015}, among others, for
some progresses in this direction. In particular, in the case of $V(x,t)=\mu v(x,t)$ for some $T$-periodic function $v\geq0,\not\equiv0$, the authors in
\cite{DT2016,DP2012} studied the limiting behaviors of the principal eigenvalue as $\mu\to\infty$ when the weight function $v$ may
be spatio-temporally degenerate (i.e. vanishes). In \cite{LLPZ2019}, we  established some monotonicity and
asymptotic behaviors of $\lambda(\alpha)$  with respect  to time period $T$.
If the advection is monotone in space $x$, i.e. $|\partial_xm|>0$, the limiting behaviors of
the principal eigenvalue for large $\alpha$ or small $D$ were investigated in \cite{PZ2015}.
However, the asymptotics of $\lambda(\alpha)$ remain open for general advection.


We also refer to \cite{CC2018, CCL2019, Hutson2001, LL2020}  for some applications of the principal eigenvalues to the evolution of dispersal in time-periodic environments.

\subsection{Main results}
Throughout this paper, we set  $D=1$ in \eqref{advection1}.
To state the main results, we introduce some notations for advection $m$.
 \begin{itemize}
   \item A {\it spatially interior critical point} of  $m$ is a point $(x_*,t_*)\in(0,1)\times[0,T]$ such that $\partial_x m(x_*,t_*)=0$, and
   it is called {\it nondegenerate} if $\partial_{xx} m(x_*,t_*)\neq0$.
   \item The boundary point $(x_*,t_*)\in\{0,1\}\times[0,T]$ is always called  spatially critical, and
   it is called {\it nondegenerate}
if either $\partial_x m(x_*,t_*)\neq0$ or $\partial_{xx} m(x_*,t_*)\neq0$.
\item A {\it point of spatially  local maximum} of $m$ is a point
$(x_*,t_*)\in[0,1]\times[0,T]$ that satisfies $m(x_*,t_*)\geq m(x,t_*)$
in a small neighborhood of $x_*$ with respect to $[0,1]$.
 \end{itemize}

Our first main result concerns the nondegenerate advection. 

\begin{theo}\label{ldnthm1}
Suppose that all spatially critical points of $m$ are nondegenerate. 
Assume $\{(\kappa_i(t),t):t\in[0,T],1\leq i\leq N\}$ with $\kappa_i\in C^1([0,T])$ is the set  of
spatially local maximum points of $m$. 
Let $\lambda(\alpha)$ be the principal eigenvalue of \eqref{advection1}. Then
\begin{equation}\label{eq:limit1}
\lim_{\alpha\rightarrow\infty}\lambda(\alpha)=\min_{1\leq i\leq N}\left\{\frac{1}{T}\int_0^T V\left(\kappa_i(s),s\right)\mathrm{d}s\right\}.
\end{equation}
\end{theo}


If  the function $m$
admits finitely many isolated maxima for every $t$, 
we conjecture that Theorem \ref{ldnthm1} remains true for  the  higher dimensional case.
Determining the asymptotic  profile of the principal eigenfunction is also
an interesting question. We   suspect that the corresponding principal eigenfunction will concentrate on some of these curves $x=\kappa_i(t)$ as $\alpha\to\infty$.

\begin{remark}\label{remark-th1} 
{\rm When $V$ and $m$ are independent of time, 
Theorem \ref{ldnthm1}  is reduced to \cite[Theorem 1.1]{CL2008} for elliptic problem \eqref{elliptic} in one-dimensional case.
If $\partial_xm>0$ (resp. $\partial_xm<0$) in $[0,1]\times[0,T]$,
it follows from Theorem {\rm\ref{ldnthm1}} that
$$\lim\limits_{\alpha\rightarrow\infty}\lambda(\alpha)=\frac{1}{T}\int_0^T V\left(1,s\right)\mathrm{d}s\ \qquad
 (\text{resp.}\ \lim\limits_{\alpha\rightarrow\infty}\lambda(\alpha)=\frac{1}{T}\int_0^T V\left(0,s\right)\mathrm{d}s),
$$
which was first established in  \cite[Theorem {\rm 1.1}]{PZ2015}.
}
\end{remark}

Our next result concerns the situation when the advection $m$ can be spatially degenerate,  e.g.
$m(x,t)$ is constant in  an interval  for each $t\in [0,T]$.
We first introduce some notations.

Let $0\leq\underline{\kappa}(t)<\overline{\kappa}(t)\leq1$ be two continuous functions defined on $[0,T]$ and $p,q\in\{\mathcal{N},\mathcal{D}\}$. Denote by $\lambda^{pq}\big((\underline{\kappa},\overline{\kappa})\big)$ the principal eigenvalue of the  problem
\begin{equation}\label{definition}
\begin{cases}
 \begin{array}{ll}
\partial_t\psi-\partial_{xx}\psi+V(x,t)\psi=\lambda\psi,\ \
&x\in(\underline{\kappa}(t),\,\overline{\kappa}(t)),\ t\in[0,T],\\
\left(\ell_1\psi+(1-\ell_1)\partial_x\psi\right)(\underline{\kappa}(t),t)=0,& t\in \left[0,T\right],\\
\left(\ell_2\psi+(1-\ell_2)\partial_x\psi\right)(\overline{\kappa}(t),t)=0,& t\in \left[0,T\right],\\
\psi(x,0)=\psi(x,T),&x\in(\underline{\kappa}(0),\overline{\kappa}(0)),
 \end{array}
\end{cases}
\end{equation}
where
  \begin{equation*}
    \ell_1=
  \left\{\begin{array}{ll}
  0,&\text{if}~ ~p=\mathcal{N},\\
  1,&\text{if}~~ p=\mathcal{D},
  \end{array}
   \right.
   \text{ and }\,\,
    \ell_2=
  \left\{\begin{array}{ll}
  0,&\text{if}~ ~q=\mathcal{N},\\
  1,&\text{if}~~ q=\mathcal{D}.
  \end{array}
  \right.
  \end{equation*}
The letters $\mathcal{N}$ and $\mathcal{D}$ represent the zero Neumann
and Dirichlet boundary conditions, respectively. The existence of $\lambda^{pq}\big((\underline{\kappa},\overline{\kappa})\big)$ can be guaranteed by the Krein-Rutman theorem \cite{KR1950}; see also \cite[Proposition 7.2]{Hess}.

Given  $\kappa_i\in C^1([0,T])$, $i=0,1,\ldots,N+1$, such that 
 \begin{equation}\label{def-a}
0=\kappa_0(t)<\kappa_1(t)<\ldots<\kappa_{N}(t)<\kappa_{N+1}(t)=1
\quad \text{for all }\,\, t\in[0,T],
 \end{equation}
we  denote
\begin{itemize}
  \item [] $\mathbf{A}=\big\{ 0\leq i\leq N:\ \ \partial_xm(x,t)<0\ \ \text{if }\,\,
      x\in(\kappa_{i}(t),\kappa_{i+1}(t))\big\};$
  \item [] $\mathbf{B}=\big\{ 0\leq i\leq N:\ \ \partial_x m(x,t)=0 \ \, \text{if }\,\,
      x\in[\kappa_{i}(t),\kappa_{i+1}(t)]\big\};$
  \item [] $\mathbf{C}=\big\{ 0 \leq i\leq N:\ \ \partial_x m(x,t)>0 \ \, \text{if }\,\,
      x\in(\kappa_{i}(t),\kappa_{i+1}(t))\big\}$.
\end{itemize}


Our second result can be stated as follows.
\begin{theo}\label{ldnthm2}
 Let  $\{\kappa_i(t)\}_{0\leq i\leq N+1}$ be functions of class $C^1$
 satisfying \eqref{def-a} and 
$\partial_xm(\kappa_i(t),t)=0$   for all $t\in[0,T]$ and $1\leq i\leq N$.
 Assume that $\{0,1,\ldots, N\}=\mathbf{A}\cup\mathbf{B}\cup\mathbf{C}$
 and that for any $0\leq i\leq N-1$, $(i, i+1)\not\in \mathbf{A}^2\cup\mathbf{B}^2\cup\mathbf{C}^2$.

Define $\mathbf{E}(\mathcal{N},\mathcal{N}), \mathbf{E}(\mathcal{N},\mathcal{D}), \mathbf{E}(\mathcal{D},\mathcal{N}), \mathbf{E}(\mathcal{D},\mathcal{D})\subset\{0,1,\ldots, N\}$ by setting
$$i\in \mathbf{E}(\mathcal{N},\mathcal{N})\Leftrightarrow \begin{cases} i-1\in \mathbf{C}, \,\, i\in \mathbf{B},\,\,i+1\in\mathbf{A} &\text{if }  1\leq i\leq N,\\
0\in \mathbf{B},\,\,1\in\mathbf{A} & \text{if }  i=0,\\
N-1\in \mathbf{C},\,\,N\in\mathbf{B} & \text{if }  i=N;\\
\end{cases}
$$
$$i\in \mathbf{E}(\mathcal{N},\mathcal{D})\Leftrightarrow \begin{cases} i-1\in \mathbf{C}, \,\, i\in \mathbf{B},\,\,i+1\in\mathbf{C} &\text{if }  1\leq i\leq N,\\
0\in \mathbf{B},\,\,1\in\mathbf{C} & \text{if }  i=0;\\
\end{cases}
$$
$$i\in \mathbf{E}(\mathcal{D},\mathcal{N})\Leftrightarrow \begin{cases} i-1\in \mathbf{A}, \,\, i\in \mathbf{B},\,\,i+1\in\mathbf{A} &\text{if }  1\leq i\leq N,\\
N-1\in \mathbf{A},\,\,N\in\mathbf{B} & \text{if }  i=N;\\
\end{cases}
$$
$$i\in \mathbf{E}(\mathcal{D},\mathcal{D})\Leftrightarrow  i-1\in \mathbf{A}, \,\, i\in \mathbf{B},\,\,i+1\in\mathbf{C}, \text{ and }  1\leq i\leq N. \quad
$$
Moreover, define the set $\mathbf{E}\subset\{0,1,\ldots, N+1\}$ by setting
$$i\in \mathbf{E}\Leftrightarrow \begin{cases} i-1\in \mathbf{C}, \,\, i\in \mathbf{A}, &\text{if }  1\leq i\leq N,\\
0\in \mathbf{A}, & \text{if }  i=0,\\
N\in\mathbf{C} & \text{if }  i=N+1.
\end{cases}
$$

Let $\lambda(\alpha)$ be the principal eigenvalue of \eqref{advection1}. Then
\begin{equation}\label{eq:limit2}
\begin{split}
\lim_{\alpha\rightarrow\infty}\lambda(\alpha)
=\min\Bigg\{&\min_{ i\in \mathbf{E}}\left[\frac{1}{T}\int_0^T V(\kappa_{i}(s),s)\mathrm{d}s\right],\ \
\min_{i\in \mathbf{E}(\mathcal{N},\mathcal{N})}\lambda^{\mathcal{N}\mathcal{N}}_i,\\
&\min_{i\in \mathbf{E}(\mathcal{N},\mathcal{D})}\lambda^{\mathcal{N}\mathcal{D}}_i,\ \ \min_{i\in \mathbf{E}(\mathcal{D},\mathcal{N})}\lambda^{\mathcal{D}\mathcal{N}}_i,\ \ \min_{i\in \mathbf{E}(\mathcal{D},\mathcal{D})}\lambda^{\mathcal{D}\mathcal{D}}_i
\Bigg\},
\end{split}
\end{equation}
where $\lambda^{pq}_i=\lambda^{pq}\big((\kappa_{i}, \kappa_{i+1})\big)$ for  $p,q\in\{\mathcal{N},\mathcal{D}\}$.
\end{theo}

\begin{remark}
{\rm If $m$ and $V$ are both  independent  of  time,
Theorem {\rm\ref{ldnthm2}} is  reduced to
\cite[Theorem {\rm 1.2}]{PZ2017} for  the elliptic eigenvalue problem \eqref{elliptic}.
However, the techniques used to prove Theorem {\rm\ref{ldnthm2}} are rather different from those in \cite{PZ2017},
due to the lack of variational characterization for the principal eigenvalue of time-dependent problem \eqref{advection1}.}
\end{remark}


\begin{remark}{\rm
When $\mathbf{B}=\emptyset$, i.e. there is no spatial degeneracy in advection,  $\mathbf{E}(\mathcal{N},\mathcal{N})= \mathbf{E}(\mathcal{N},\mathcal{D})= \mathbf{E}(\mathcal{D},\mathcal{N})= \mathbf{E}(\mathcal{D},\mathcal{D})=\emptyset$ and  Theorem \ref{ldnthm2} is a slightly stronger version of Theorem \ref{ldnthm1} without the  nondegeneracy
assumption on the  spatial critical points of $m$.
See Remark \ref{rem2.1} for two examples.
When $\mathbf{B}\neq\emptyset$,
if we let the strip  $(\kappa_i(t),\,\kappa_{i+1}(t))$ shrinks to some curve $\tilde{\kappa}_i(t)$
for every  $i\in \mathbf{B}$,
then
  $$
  \lambda^{\mathcal{D}\mathcal{D}}\big((\kappa_i,\kappa_{i+1})\big)\to\infty, \ \
 \lambda^{\mathcal{D}\mathcal{N}}\big((\kappa_i,\kappa_{i+1})\big)\to\infty, \ \
  \lambda^{\mathcal{N}\mathcal{D}}\big((\kappa_i,\kappa_{i+1})\big)\to\infty,
$$
$$ \lambda^{\mathcal{N}\mathcal{N}}\big((\kappa_i,\kappa_{i+1})\big)\to
  \frac{1}{T}\int_0^T V(\tilde{\kappa}_i(s),s)\,\mathrm{d}s.$$
Hence, \eqref{eq:limit2}
is reduced to \eqref{eq:limit1},
i.e.
Theorem {\rm \ref{ldnthm2}}
coincides with Theorem {\rm \ref{ldnthm1}}.
}
\end{remark}

\begin{remark}\label{re-th2}
{\rm
We further illustrate Theorem {\rm\ref{ldnthm2}} 
by 
the special case 
\begin{equation}\label{abvection_m}
    m(x,t)=m_1(x)m_2(t),
\end{equation}
where $m_2\in C([0,T])$ is $T$-periodic and $m_1\in C^{2}([0,1])$.
In the context of Theorem {\rm\ref{ldnthm2}}, 
assume that $m_2>0$ on $[0,T]$
and the graph of $m_1$ is given as in Fig. \ref{figure1_1} with the set of critical points $\{\kappa_i: 0\leq i\leq 10\}$.
Then $\mathbf{A}=\{0, 4, 6\}$, $\mathbf{B}=\{2, 5, 7, 9\}$, and $\mathbf{C}=\{1, 3, 8\}$. Using the notations in Theorem {\rm\ref{ldnthm2}},
$\mathbf{E}=\{0, 4\}$, $ \mathbf{E}(\mathcal{N},\mathcal{N})=\{9\}$,  $\mathbf{E}(\mathcal{N},\mathcal{D})=\{2\}$, $\mathbf{E}(\mathcal{D},\mathcal{N})=\{5\}$, and $\mathbf{E}(\mathcal{D},\mathcal{D})=\{7\}$. Then 
Theorem {\rm\ref{ldnthm2}} yields
\begin{equation*}
\begin{split}
\lim_{\alpha\rightarrow\infty}\lambda(\alpha)
=\min\Bigg\{&\frac{1}{T}\int_0^T V(0,s)\mathrm{d}s,\ \ \frac{1}{T}\int_0^T V(\kappa_4,s)\mathrm{d}s, \ \
\lambda^{\mathcal{N}\mathcal{N}}\big((\kappa_9, \kappa_{10})\big),\\
&\lambda^{\mathcal{N}\mathcal{D}}\big((\kappa_2, \kappa_3)\big),\ \ \lambda^{\mathcal{D}\mathcal{N}}\big((\kappa_5, \kappa_6)\big),\ \ \lambda^{\mathcal{D}\mathcal{D}}\big((\kappa_7, \kappa_8)\big)
\Bigg\}.
\end{split}
\end{equation*}
We also refer to Propositions {\rm \ref{ldnlem2}} and {\rm \ref{ldnlem3}} for further details.
}
\begin{figure}[http!!]
  \centering
\includegraphics[height=2.2in]{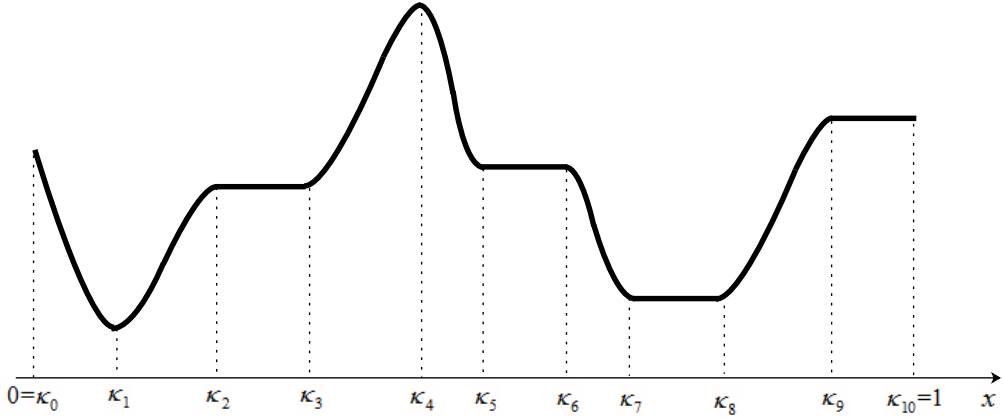}
  \caption{\small  
  An example of $m_1(x)$ with
  degenerate critical points
  in Remark \ref{re-th2}, where $\kappa_i$ ($0\leq i\leq 10$) denotes  its critical point.}\label{figure1_1}
  \end{figure}
\end{remark}

Our next result concerns  the effect of temporally degenerate advection
on the limit of $\lambda(\alpha)$ as $\alpha\rightarrow\infty$. To emphasize the ideas
of our proof and also make the presentation of our result clearer, 
assume $\partial_xm(x,t)=b(t)$ for some  $T$-periodic function $b$, where $b$ allows to vanish somewhere
and is referred as the temporal degeneracy. 
Then problem \eqref{advection1} becomes
\begin{equation}\label{LTD_eq 1}
\begin{cases}
 \begin{array}{ll}
 \smallskip
\partial_{t}\varphi-\partial_{xx}\varphi-\alpha b(t)\partial_{x}\varphi+V(x,t)\varphi=\lambda \varphi\ \ & \text{in}\,\,(0,1)\times[0,T],\\
\smallskip
\partial_{x}\varphi(0,t)=\partial_{x}\varphi(1,t)=0&{\text{on}}\,\,[0,T],\\
\varphi(x,0)=\varphi(x,T) &{\text{on}}\,\,(0,1).
 \end{array}
\end{cases}
\end{equation}

Hereafter, for any $\phi\in C([0,1]\times \mathbb{R})$, we use the notations $\phi(\cdot,t^+)$ and $\phi(\cdot,t^-)$
to represent the right and left limit of $\phi$ at time $t$ respectively; similarly,
the notations $\phi(x^+,\cdot)$ and $\phi(x^-,\cdot)$ mean the right and left limit
at spatial location $x$.
\begin{Thm}\label{LTD_thm main}
Given any sequence $\{t_i\}_{0\leq i\leq N+1}$
with $0=t_0<t_1<\ldots<t_{N+1}=T$, denote
\begin{equation*}
  \begin{array}{c}
  \mathbb{A}=\big\{0\leq i\leq N:\ \,b(t)<0 \,\text{ in } (t_{i},t_{i+1})\big\},\\
  \mathbb{B}=\big\{0\leq i\leq N:\ \, b(t)\equiv0\, \text{ on } [t_{i},t_{i+1}]\big\},\\
  \mathbb{C}=\big\{0\leq i\leq N:\ \, b(t)>0\, \text{ in } (t_{i},t_{i+1})\big\}.
  \end{array}
\end{equation*}
Assume
$\mathbb{A}\cup\mathbb{B}\cup\mathbb{C}=\{0,\ldots, N\}$. Let $\lambda(\alpha)$ be the principal eigenvalue of \eqref{LTD_eq 1}. Then
$\lambda(\alpha)\to\lambda_\infty$ as $\alpha\rightarrow\infty$,
where $\lambda_\infty$ is the principal eigenvalue of the problem

\begin{equation}\label{LTD_auxi main}
\begin{cases}
\left.\begin{array}{ll}
\smallskip
 \partial_{t}\psi+V(0,t)\psi=\lambda \psi\ \ \ \ &{\text{in}}\,\,(0,1)\times\left(t_i,t_{i+1}\right],  \\
             \psi(x,t_{i}^+)\equiv\psi(0,t_{i}^-)&{\text{on}}\,\,(0,1),\\
              \end{array} \right\}\ i\in \mathbb{A} \\
\left.\begin{array}{ll}
\smallskip
 \partial_{t}\psi-\partial_{xx}\psi+V(x,t)\psi=\lambda\psi&{\text{in}}\,\,(0,1)\times\left(t_i,t_{i+1}\right],  \\
              \psi(x,t_{i}^+)=\psi(x,t_{i}^-)&{\text{on}}\,\,(0,1),\\
             \end{array}\right\}\ i\in \mathbb{B}\\
\left.\begin{array}{ll}
\smallskip
 \partial_{t}\psi+V(1,t)\psi=\lambda \psi\ \ \ \ \ &{\text{in}}\,\,(0,1)\times\left(t_i,t_{i+1}\right],  \\
             \psi(x,t_{i}^+)\equiv\psi(1,t_{i}^-)&{\text{on}}\,\,(0,1),\\
              \end{array}\right\}\ i\in \mathbb{C}\\
\,\,\,\partial_{x}\psi(0,t)=\partial_{x}\psi(1,t)=0\,\,\ {\text{on}}\,\,\left[0,T\right],\\
\,\,\,\psi(x,0)=\psi(x,T)\,\,\,\,\,\quad\quad\quad{\text{on}}\,\,(0,1).
\end{cases}
\end{equation}
\end{Thm}

The existence and uniqueness of the principal eigenvalue  
for problem \eqref{LTD_auxi main} is proved in Proposition
\ref{principaleigenfunction}, which is shown to be real and simple.
Some typical examples included by Theorem \ref{LTD_thm main} are provided
in Propositions \ref{LTD_thm 1}-\ref{LTD_thm 3}. Theorem \ref{LTD_thm main} shows that for $i\in\mathbb{B}$, i.e. there is no advection in
time interval $\left[t_i,t_{i+1}\right]$,  the whole space  $[0,1]$ turns out to influence the asymptotic behaviors of
 principal eigenvalue, while for $i\in \mathbb{A}$ (resp. $i\in \mathbb{C}$), only the boundary points in $\{(0, t): \, t\in[0,T]\}$ (resp. $\{(1, t): \, t\in[0,T]\}$) matter. In particular, when $\mathbb{B}=\emptyset$ and $\mathbb{A}\cup\mathbb{C}=\{0,\ldots, N\}$, we can deduce from Theorem \ref{LTD_thm main} and the definition of $\lambda_\infty$ in \eqref{LTD_auxi main} that
$$\lim_{\alpha\rightarrow\infty}\lambda(\alpha)=\frac{1}{T}\left[\sum_{i\in \mathbb{A}}\int^{t_{i+1}}_{t_i}V(0,s)\mathrm{d}s+\sum_{i\in \mathbb{C}}\int^{t_{i+1}}_{t_i}V(1,s)\mathrm{d}s\right].$$

\begin{remark}\label{rem_4.2}
{\rm In addition to the temporally degenerate case considered in Theorem {\rm\ref{LTD_thm main}},
we can in fact  deal with some more general temporal degeneracy. 
For example,  assume that there exist  constants $\kappa_*\in(0,1)$ and $t_*\in(0,T)$ such that 
\begin{equation}\label{examp1}
\begin{cases}
\partial_x m(x,t)>0 \ \ &\text{in } \, \left[0,\kappa_*\right)\times(0,t_*),\\
\partial_x m(x,t)<0 &\text{in } \, \left(\kappa_*,1\right]\times(0,t_*),\\
\partial_x m(\kappa_*,t)=0 &  \text{on } \, (0,t_*),\\
\partial_x m(x,t)=0 & \text{in } \,
\left[0,1\right]\times \left[t_*,T\right].
\end{cases}
\end{equation}
See Fig. {\rm\ref{figure6_a}} for the profile of  $\partial_x m$.
We may use the similar arguments in Proposition \ref{LTD_thm 1} to  show $\lambda(\alpha)\to \lambda_{*}$ as $\alpha\rightarrow\infty$,
where $\lambda_{*}$ denotes the principal eigenvalue of the problem
$$
\left\{\begin{array}{ll}
\smallskip
\partial_{t}\psi+V(\kappa_*,t)\psi=\lambda\psi & {\mathrm{in}}\, (0,1)\times(0,t_*),\\
\smallskip
\partial_{t}\psi-\partial_{xx}\psi+V(x,t)\psi=\lambda\psi
& {\mathrm{in}}\, (0,1)\times \left[t_*,T\right],\\
\smallskip
\psi(x,t_*^{+})=\psi(x,t_*^{-})&{\mathrm{ on}}\,(0,1),\\
\smallskip
\partial_{x}\psi(0,t)=\partial_{x}\psi(1,t)=0& {\mathrm{on}}\, \left[0,T\right],\\
\psi(x,0)=\psi(x,T)&{\mathrm{on}}\,(0,1),
\end{array}
\right.$$
where the existence and uniqueness of $\lambda_{*}$
can be proved  as in Proposition {\rm\ref{principaleigen}}. 
Observe that the limit value $\lambda_*\to \frac{1}{T}\int_0^T V(\kappa_*,s)\mathrm{d}s$ as $t_* \nearrow T$,  in agreement with  the conclusions 
of Theorem {\rm \ref{ldnthm1}}.
}

\begin{figure}[http!!]
  \centering
\includegraphics[height=1.8in]{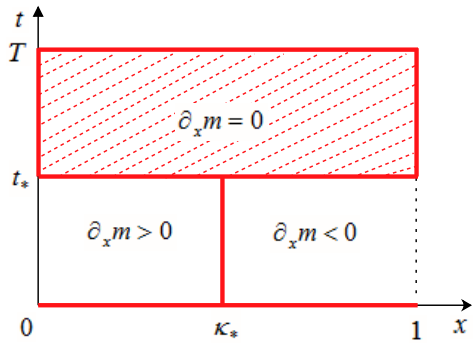}
  \caption{\small  Profile of  $\partial_x m$ defined by \eqref{examp1}, where the red shaded area and  red  solid lines together represent the set $\{(x,t)\in[0,1]\times[0,T]:\partial_x m(x,t)=0\}$.}\label{figure6_a}
  \end{figure}

\end{remark}

\subsection{Discussion}
The spatially  
and temporally degenerate advection  are separately considered in  Theorems {\rm\ref{ldnthm2}} and  {\rm\ref{LTD_thm main}}. In fact,  our ideas 
in the paper can deal with the case when the advection $m$ possesses both spatial
and temporal degeneracy. Below we provide an example as an illustration.
Assume  $m$ takes the form of \eqref{abvection_m}, and there are
constants $0<\kappa_1<\kappa_2<1$ and $0<t_*<T$ such that
$$ m_2(t)>0\ \text{ in }\,\left(0,t_*\right)\, \text{ and }
\,m_2(t)\equiv0\ \text{ in }\,\left[t_*,T\right],\ \ \text{(Temporal degeneracy)},$$
$$m_1'(x)>0\,\text{ in  }\,\left[0,\kappa_1\right),\quad m_1'(x)\equiv0\text{ in }\,[\kappa_1,\kappa_2]\,\,
\text{ and }\,\, m_1'(x)<0\,\text{ in }\,\left(\kappa_2,1\right],\ \ \text{(Spatial degeneracy)}.$$
See  Fig. {\rm\ref{figure6_b}} for the profile of the spatio-temporally degenerate $\partial_x m$. 
Combining the proofs of Propositions  {\rm\ref{ldnlem2}} and {\rm\ref{LTD_thm 1}},
one can prove that $\lambda(\alpha)\to \tilde\lambda_*$ as $\alpha\to\infty$, where $\tilde\lambda_*$
denotes the principal eigenvalue of the problem
\begin{equation}\label{LTD_auxi}
\left\{\begin{array}{ll}
\smallskip
\partial_{t}\psi-\partial_{xx}\psi+V\psi=\lambda\psi\ \
&{\mathrm{in}}\,\,[\kappa_1,\kappa_2]\times\left(0,t_{*}\right],  \\
\smallskip
\partial_{x}\psi(\kappa_1,t)=\partial_{x}\psi(\kappa_2,t)=0
&{\mathrm{on}}\,\,\left(0,t_{*}\right],\\
\smallskip
 \psi(x,t_{*}^+)=\hat{\psi}(x,t_{*})&{\mathrm{on}}\,\,(0,1),\\
 \smallskip
\partial_{t}\psi-\partial_{xx}\psi+V\psi=\lambda\psi
&{\mathrm{in}}\,\,(0,1)\times\left(t_{*},T\right], \\
\smallskip
\partial_{x}\psi(0,t)=\partial_{x}\psi(1,t)=0&{\mathrm{on}}\,\,\left(t_{*},T\right],\\
\psi(x,0)=\psi(x,T)&{\mathrm{on}}\,\,(0,1),
\end{array}\right.
\end{equation}
where $\hat{\psi}(x,t_{*})$ is the extension of $\psi(x,t_{*})$ by setting
$\hat{\psi}(x,t_{*})=\psi(\kappa_1,t_{*})$ for $x\in[0,\kappa_1)$ and
$\hat{\psi}(x,t_{*})=\psi(\kappa_2,t_{*})$ for $x\in(\kappa_2, 1]$.
The existence and uniqueness of $\tilde\lambda_*$ 
follows from the same arguments as in Proposition {\rm\ref{principaleigen}}.
\begin{figure}[http!!]
  \centering
\includegraphics[height=1.8in 
]{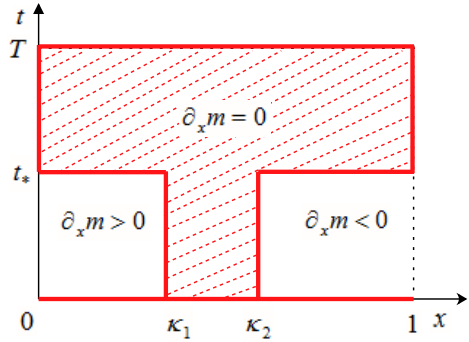}
  \caption{\small  Profile of  $\partial_x m$ for the spatio-temporally degenerate advection, where the red shaded area and red solid lines together represent the set $\{(x,t)\in[0,1]\times[0,T]:\partial_x m(x,t)=0\}$.}\label{figure6_b}
  \end{figure}

For general advection $m$, the cases of spatio-temporal degeneracies
are so many that we can not address 
them specifically 
or state the results in a general theorem. This is merely a technical
point which does not involve many new ideas, and thus
is left to the interested reader. Moreover, the techniques
developed in this paper can be used to investigate the
asymptotic behaviors of the principal eigenvalue 
 for \eqref{advection1} subject to other boundary conditions, including the zero
Dirichlet boundary conditions and the Robin boundary conditions. 

Our proofs in this paper  rely heavily upon the construction of (almost) optimal pairs of
sub-solutions and super-solutions in the sense of Definition \ref{appendixldef},
and applications of Proposition \ref{appendixlem}. 
To clarify the ideas, instead of proving Theorems \ref{ldnthm1}-\ref{LTD_thm main}
directly, we shall provide the detailed proof for some typical examples
and the three main theorems follow by a similar argument.

\subsection{Organization of the paper}
In Sect. \ref{preliminary},  we present the theory of  weak sub-solutions and super-solutions.
In Sect. \ref{Nondegenerate},
we discuss the nondegenerate advection and prove Theorem \ref{ldnthm1}.
Sect. \ref{Spatially_degenerate} concerns spatially degenerate advection and Theorem \ref{ldnthm2}
is proved there. The temporally degenerate advection is considered in Sect. \ref{S4} and
Theorem \ref{LTD_thm main} is proved,
which requires more delicate analysis.
Finally,  we verify the existence and uniqueness of the principal eigenvalue for
problem \eqref{LTD_auxi main} in Appendix \ref{appenB}, and prove \eqref{ldnaux} in Appendix \ref{appenC}.


\section{Generalized super/sub-solution for a time-periodic parabolic operator}\label{preliminary}
In this section, we introduce the  definition of super/sub-solution for a time-periodic parabolic operator, and then
establish the relation of positive super/sub-solution and the sign of principal
eigenvalue of an associated eigenvalue problem. This result is a generalization of Proposition 2.1 and Corollary 2.1
in \cite{PZ2015} in one-dimensional case, and plays a vital role in the establishment
of main findings in the present paper.

Let $\mathcal{L}$ denote the following linear parabolic operator on $(0,1)\times\mathbb{R}$:
\begin{equation}\label{liu-01}
    \mathcal{L}=\partial_{t}\varphi-a_1(x,t)\partial_{xx}- a_2(x,t)\partial_x+a_0(x,t).
\end{equation}
We always assume $a_1(x,t)>0$ so that $\mathcal{L}$ is uniformly elliptic for each $t\in\mathbb{R}$, and
assume $a_0,a_1,a_2\in C([0,1]\times\mathbb{R})$ are $T$-periodic with respect to $t$.

Consider the linear parabolic problem
\begin{equation}\label{appendix}
 \left\{\begin{array}{ll}
 \smallskip
\mathcal{L}\varphi=0\ \ &{\text{in}}\,\,(0,1)\times[0,T],\\
\smallskip
 \partial_x\varphi(0,t)=\partial_x\varphi(1,t)=0 \ \ & {\text{on}}\,\,[0,T], \\
 \varphi(x,0)=\varphi(x,T) &{\text{on}}\,\,(0,1).
 \end{array}
 \right.
 \end{equation}
We now give the definition of super/sub-solution corresponding to \eqref{appendix}. 

\begin{definition}\label{appendixldef}
The lower semi-continuous function $\overline{\varphi}$ in $[0,1]\times[0,T]$ is  called a  super-solution of \eqref{appendix} if there exist sets $\mathbb{X}$ and $\mathbb{T}$ consisting of at most finitely many continuous curves: 
$$\mathbb{X}=\emptyset \text{ or }\mathbb{X}=\left\{(\kappa_i(t),t):\ \,t\in(0,T),\ i=1,\ldots,N\right\},$$
$$\mathbb{T}=\emptyset \text{ or }\mathbb{T}=\left\{(x,\tau_i(x)):\ \,x\in(0,1),\ i=1,\ldots,M\right\},$$
for some integers $N, M\geq1$, where the continuous functions $\kappa_i:\ [0,T]\mapsto (0,1)$ and $\tau_i:\ [0,1]\mapsto (0,T)$ are such that
\smallskip

\noindent{ } {\rm(1)} $\overline{\varphi}\in C\left((0,1)\times(0,T)\setminus\mathbb{T}\right)\cap C^2\left((0,1)\times(0,T)\setminus\left(\mathbb{X}\cup\mathbb{T}\right)\right)$;

\smallskip

\noindent{ } {\rm(2)} $\partial_x\overline{\varphi}(x^+,t)<\partial_x\overline{\varphi}(x^-,t)$,\ $\forall (x,t)\in\mathbb{X};$

\smallskip

\noindent{  } {\rm(3)}
$\overline{\varphi}(x,t)=\overline{\varphi}(x,t^-)<\overline{\varphi}(x,t^+)$ or $\overline{\varphi}(x,t^-)=\overline{\varphi}(x,t^+)$ but $\partial_t\overline{\varphi}(x,t^+)<\partial_t\overline{\varphi}(x,t^-)$, $\forall (x,t)\in\mathbb{T};$

\smallskip

\noindent{ } {\rm(4)}
 $\overline{\varphi}$ satisfies
  \begin{equation}\label{def}
 \left\{\begin{array}{ll}
 \smallskip
\mathcal{L}\overline{\varphi}\geq 0\ \ &{\mathrm{ in }}\,\,((0,1)\times(0,T))\setminus\left(\mathbb{X}\cup\mathbb{T}\right),\\
\smallskip
 \partial_x\overline{\varphi}(0,t)\leq0,\ \ \partial_x\overline{\varphi}(1,t)\geq0  &{\mathrm{ on }}\,\,[0,T], \\
 \overline\varphi(x,0)\geq\overline\varphi(x,T) &\mathrm{ on } \,\,(0,1).
 \end{array}
 \right.
 \end{equation}
A super-solution $\overline{\varphi}$ is called  a strict super-solution if it is not a solution of \eqref{appendix}. Moreover,
a function $\underline{\varphi}$ is called a (strict) sub-solution of \eqref{appendix} if $-\underline{\varphi}$ is a  (strict) super-solution.
\end{definition}

\begin{lemma}\label{appendixlem}
Let $\overline{\varphi}\geq0$ be a super-solution of \eqref{appendix} defined in Definition {\rm\ref{appendixldef}}. Then $\overline{\varphi}>0$ in $[0,1]\times[0,T]$ unless $\overline{\varphi}\equiv0$.
\end{lemma}

\begin{proof}
Assume that $\overline{\varphi}\not\equiv0$ and we shall prove $\overline{\varphi}>0$
in $[0,1]\times[0,T]$. First, by the Hopf's boundary lemma for parabolic equations (see e.g. \cite[Proposition 13.3]{Hess}), we deduce  $\overline{\varphi}>0$
on the boundary $\{0,1\}\times[0,T]$.

Suppose  the assertion fails, then there exists some interior point $(x_0,t_0)\in(0,1)\times[0,T]$ such that
\rm{(i)} $\overline{\varphi}(x_0,t_0)=0$ and \rm{(ii)} $x_0\in\partial\{x\in \Omega: \overline{\varphi}(x,t_0)=0\}$.
This is possible since  $\overline{\varphi}>0$
on the boundary $\{0,1\}\times[0,T]$. 
We mention that such point $(x_0,t_0)$ can be chosen such that $t_0>0$. Indeed, if $t_0=0$, the fact that $\overline\varphi(x_0,0)\geq\overline\varphi(x_0,T)$ implies $\overline\varphi(x_0,T)=0$. Then we can select some $(\tilde{x}_0,T)$, satisfying \rm{(i)} and \rm{(ii)}, to replace $(x_0,t_0)$.


Next, we claim that $(x_0,t_0)\not\in\mathbb{X}\cup\mathbb{T}$. If $(x_0,t_0)\in\mathbb{X}$, by {\rm{(2)}} in Definition \ref{appendixldef}, it holds that $\partial_x\overline{\varphi}(x_0^+,t_0)<\partial_x\overline{\varphi}(x_0^-,t_0)\leq0$. Since $\overline{\varphi}(x_0,t_0)=0$, this  contradicts $\overline{\varphi}\geq0$. If $(x_0,t_0)\in\mathbb{T}$, 
in light of {\rm{(3)}} in Definition \ref{appendixldef},
we  distinguish two cases:\ {\rm{(a)}} $\overline{\varphi}(x_0,t_0^-)<\overline{\varphi}(x_0,t_0^+)$ and
{\rm{(b)}}  $\overline{\varphi}(x_0,t_0^-)=\overline{\varphi}(x_0,t_0^+)$ to reach a contradiction.
When {\rm{(a)}} occurs, by  {\rm{(3)}}, $\overline{\varphi}(x_0,t_0^-)=\overline{\varphi}(x_0,t_0)=0$.  Choose  $\delta_1\in (0,\delta)$ small such that $B_{\delta_1}^+:=B_{\delta_1}(x_0,t_0)\cap((0,1)\times[0,t_0))$ satisfies $B_{\delta_1}^+\cap (\mathbb{X}\cup\mathbb{T})=\emptyset$,
$$\mathcal{L}\overline{\varphi}\geq 0\,\, \text{ in }\,\,B_{\delta_1}^+ \quad \text{and}\quad\overline{\varphi}(x_0,t_0^-)=0.$$
We may apply the classical strong maximum principle for parabolic equations (see e.g. Proposition 13.1 and Remark 13.2 in \cite{Hess}) to arrive at
$\overline{\varphi}\equiv0$ in $\overline{B}_{\delta_1}^+$. 
This contradicts to \rm{(ii)}.
When {\rm{(b)}} occurs, by our assumption, $\partial_t\overline{\varphi}(x_0,t_0^+)<\partial_t\overline{\varphi}(x_0,t_0^-)\leq0$, which together with {\rm(i)} contradicts $\overline{\varphi}\geq0$.
Hence, $(x_0,t_0)\not\in\mathbb{X}\cup\mathbb{T}$.

Therefore, there exists some $\delta_2\in (0,\delta)$ such that $B_{\delta_2}(x_0,t_0)\cap\left(\mathbb{X}\cup\mathbb{T}\right)=\emptyset$ and
$$\mathcal{L}\overline{\varphi}\geq 0\,\, \text{ in}\,\,B_{\delta_2}(x_0,t_0) \,\, \text{ and }\,\,\overline{\varphi}(x_0,t_0)=0.$$
Due to $t_0>0$, by the classical strong maximum principal, we can conclude that $\overline{\varphi}\equiv0$ in $B_{\delta_2}(x_0,t_0)\cap([0,1]\times[0,t_0])$, a contradiction to \rm{(ii)} again. The proof is now complete.
\end{proof}

Consider the following eigenvalue problem:
\begin{equation}\label{A.3}
 \left\{\begin{array}{ll}
 \medskip
\mathcal{L}\varphi=\lambda(\mathcal{L})\varphi\ \ &{\text{in}}\,\,(0,1)\times[0,T],\\
\medskip
 \partial_x\varphi(0,t)=\partial_x\varphi(1,t)=0 \ \ & {\text{on}}\,\,[0,T], \\
 \varphi(x,0)=\varphi(x,T) &\text{on}\,\,(0,1).
 \end{array}
 \right.
 \end{equation}
The main result in this section can be stated as follows.

\begin{prop}\label{appendixprop}
Let $\lambda(\mathcal{L})$ denote the principal eigenvalue of \eqref{A.3}.
If \eqref{appendix} admits some strict positive super-solution defined in Definition {\rm \ref{appendixldef}}, then
$\lambda(\mathcal{L})\geq0$. Moreover, if \eqref{appendix} admits some strict  nonnegative sub-solution  defined in Definition {\rm \ref{appendixldef}}, then $\lambda(\mathcal{L})\leq0$.
\end{prop}

Thanks to Lemma \ref{appendixlem}, Proposition \ref{appendixprop} can be proved by
 the same arguments as in Proposition 2.1 and Corollary 2.1  of \cite{PZ2015}, where the equivalent relationship of $\lambda(\mathcal{L})>0$ and a positive strict (classical) super-solution, is established. Such an idea  can be traced back to Walter \cite{W1989} for the elliptic operators under zero Dirichlet boundary conditions. We omit the details
and refer interested readers to \cite{PZ2015}.

\section{Nondegenerate advection: proof of Theorem \ref{ldnthm1}}\label{Nondegenerate}
This section is devoted to the proof of Theorem \ref{ldnthm1}
which concerns the case that
all critical points of advection function $m$ are nondegenerate.
The results for a typical example will be proved first, and then
Theorem \ref{ldnthm1} follows from a similar argument. 
From now on, we shall use
$\mathcal{L}_{\alpha}$ to denote the following time-periodic parabolic operator:
$$\mathcal{L}_{\alpha}:=\partial_{t}-\partial_{xx}-
\alpha \partial_xm\cdot\partial_x+V,\ \ (\alpha\geq0).$$

Our first result concerns the following special case:

\begin{prop}\label{ldnlem1}
Suppose that there exist $T$-periodic functions $\kappa_i\in C^1(\mathbb{R})$ ($i=1,2,3$) such that
$0\leq \kappa_1(t)<\kappa_2(t)<\kappa_3(t)\leq 1$ for all $t\in[0,T]$ and 
 \begin{equation}\label{liu011}
    \begin{cases}
\partial_x m(x,t)>0,\ \ &  x
\in\left(0,\kappa_1(t)\right)\cup\left(\kappa_2(t),\kappa_3(t)\right),\,\,\, t\in [0,T],\\
\partial_x m(x,t)=0,\ \ &  x=\kappa_1(t),\,\kappa_2(t),\,\kappa_3(t), \,\,\, t\in [0,T],\\
\partial_x m(x,t)<0,\ \ &  x
\in \left(\kappa_1(t), \kappa_2(t)\right)\cup\left(\kappa_3(t), 1\right), \,\,\, t\in[0,T]. 
\end{cases}
 \end{equation}
Let $\lambda(\alpha)$ be the principal eigenvalue of \eqref{advection1}. Then
 $$
 \lim_{\alpha\rightarrow\infty}\lambda(\alpha)
 =\min\left\{\frac{1}{T}\int_0^T V\left(\kappa_1(s),s\right)\mathrm{d}s,\ \
 \frac{1}{T}\int_0^T V\left(\kappa_3(s),s\right)\mathrm{d}s\right\}.
 $$
\end{prop}

\begin{proof} Without loss of generality,  assume that
\begin{equation}\label{a-a}
 \int_0^T V\left(\kappa_1(s),s\right)\mathrm{d}s
 \leq\int_0^T V\left(\kappa_3(s),s\right)\mathrm{d}s.
\end{equation}
It suffices to show
 $\lambda(\alpha)\to\frac{1}{T}\int_0^T V\left(\kappa_1(s),s\right)\mathrm{d}s
 $ as $\alpha\rightarrow\infty$.
We divide the proof into two steps.

 \smallskip
\noindent {\bf Step 1. } We first prove
\begin{equation}\label{ldn1}
 \liminf_{\alpha\to\infty}\lambda(\alpha)
 \geq\frac{1}{T}\int_0^T V\left(\kappa_1(s),s\right)\mathrm{d}s.
\end{equation}
We shall construct a  positive super-solution
$\overline \varphi$ in the sense of Definition \ref{appendixldef}. 
More precisely, for any given small constant $\epsilon>0$,
we devote ourselves to identifying the curve set $\mathbb{X}$
there and constructing a continuous function $\overline \varphi>0$  such that
 \begin{equation} \label{ldn2}
 \left\{\begin{array}{ll}
 \medskip
 \mathcal{L}_\alpha\overline{\varphi}\geq\left[{1\over T}\int_0^T V\left(\kappa_1(s),s\right)\mathrm{d}s-\epsilon\right]\overline{\varphi}
 &\text {in}\,\,\left((0,1)\times[0,T]\right)\setminus\mathbb{X},\\
 \medskip
 \partial_x\overline{\varphi}(0,t)\leq0, \ \
 \partial_x\overline{\varphi}(1,t)\geq0  &\text {on}\,\,[0,T], \\
 \medskip
 \overline{\varphi}(x,0)=\overline{\varphi}(x,T) &\text{on}\,\,(0,1),\\
 \partial_{x}\overline{\varphi}(x^+,t)<  \partial_{x}\overline{\varphi}(x^-,t) &\text{in}\,\,\mathbb{X},
 \end{array}
 \right.
\end{equation}
provided that $\alpha>0$ is sufficiently large. Once such a super-solution $\overline \varphi$ exists,
a direct application of Proposition \ref{appendixprop}  to the operator
 $\mathcal{L}:=\mathcal{L}_\alpha-\left[{1\over T}\int_0^T V\left(\kappa_1(s),s\right)\mathrm{d}s-\epsilon\right]$
yields \eqref{ldn1}. 


To this end, we introduce some constant $\delta$ with $0<\delta\ll1$ such that
\begin{equation}\label{liu009}
|V(x,t)-V\left(\kappa_i(t),t\right)|<\epsilon/2\ \ \text{ on } \Delta_{i\delta}:=\{(x,t)\in[0,1]\times[0,T]: |x-\kappa_i(t)|<\delta\},\,\, i=1,2,3.
\end{equation}
We construct  $\overline{\varphi}$ on different  regions of $[0,1]\times[0,T]$.

\smallskip
$\mathbf{(1)}$ For $(x,t)\in \Delta_{1\delta}$, 
we define 
\begin{equation}\label{liu010}
    \overline{\varphi}(x,t):=\overline{z}_{1}(x,t)f_{1}(t),
\end{equation}
where $f_{1}$ is a $T$-periodic function given by
 \begin{equation}\label{ldn7}
 \begin{array}{ll}
 \displaystyle
 f_{1}(t)=\exp\left[{-\int_0^t V\left(\kappa_1(s), s\right)\mathrm{d}s+{t\over T}\int_0^T V\left(\kappa_1(s), s\right)\mathrm{d}s}\right].
 \end{array}
 \end{equation}
To verify the first inequality in \eqref{ldn2},  by \eqref{liu010} we need to find $\overline{z}_{1}>0$ such that 
 \begin{equation} \label{ldn3}\begin{array}{ll}
\partial_{t}\overline{z}_{1}-\partial_{xx}\overline{z}_{1}-\alpha \partial_x m\partial_x\overline{z}_{1}+\left[V-V\left(\kappa_1,t\right)
 +\epsilon\right]\overline{z}_{1}> 0 \quad \text{ on } \,\Delta_{1\delta}.
 \end{array}
 \end{equation}
Indeed, set
 $$\overline{z}_{1}(x,t):=M_{1}+\left(x-\kappa_1(t)\right)^2,$$
 where $M_1$ is chosen such that
 $\epsilon M_1>4+4\delta \| \kappa_1'\|_\infty$.  
By our assumption \eqref{liu011}, it is noted that 
 $$
 -\alpha \partial_x m \cdot \partial_x \overline{z}_{1}=-2\alpha \left(x-\kappa_1(t)\right)\partial_x m\geq0 \quad \text{ on } \,\Delta_{1\delta},
 $$
 and by \eqref{liu009}, $V(x,t)-V\left(\kappa_1,t\right)
 +\epsilon\geq \epsilon/2$.
Hence, by the choice of $M_{1}$ we can verify \eqref{ldn3}, 
so that the chosen $\overline{\varphi}$ in \eqref{liu010} satisfies
 \eqref{ldn2} on $\Delta_{1\delta}$ for all $\alpha>0$.

\smallskip
$\mathbf{(2)}$ For $(x,t)\in \Delta_{3\delta}$ 
as above,  we construct
 $$\overline{\varphi}(x,t):=\overline{z}_{3}(x,t)f_{3}(t),
 $$
with $T$-periodic function $f_{3}$ given by
 \begin{equation*}
 \begin{array}{ll}
 \displaystyle
 f_{3}(t)=\exp\left[{-\int_0^tV\left(\kappa_3(s), s\right)\mathrm{d}s+{t\over T}\int_0^T V\left(\kappa_3(s), s\right)\mathrm{d}s}\right].
 \end{array}
\end{equation*}
Set $\overline{z}_{3}(x,t):=M_{3}+\left(x-\kappa_3(t)\right)^2$.
It can be checked that for large $M_{3}$ (independent of $\alpha>0$),
 \begin{equation*}\begin{array}{ll}
\partial_{t}\overline{z}_{3} -\partial_{xx}\overline{z}_{3}-\alpha \partial_x m\partial_{x}\overline{z}_{3}+\left[V-V\left(\kappa_3,t\right)
 +\epsilon\right]\overline{z}_{3}>0 \quad \text{ on } \,\Delta_{3\delta},
 \end{array}
 \end{equation*}
which together with \eqref{a-a} implies that the constructed $\overline{\varphi}$
satisfies \eqref{ldn3} on $\Delta_{3\delta}$ for any $\alpha>0$. 

\smallskip

$\mathbf{(3)}$
For $(x,t)\in \Delta_{2\delta}$, 
we  construct $\overline \varphi$  in the form of
 \begin{equation*}
\overline \varphi(x,t):=M_{2}\left[1-\frac{\left(x-\kappa_2(t)\right)^2}{2\delta^2}\right],
  \end{equation*}
where we may choose $\delta>0$ further small if necessarily such that 
\begin{equation}\label{liu004}
\delta \|\kappa'_2\|_\infty+\delta ^2
\left\|V-{1\over T}\int_0^T V\left(\kappa_1(s),s\right)\mathrm{d}s+\epsilon\right\|_\infty<1.
  \end{equation}
  and the constant  $M_{2}$ is chosen to satisfy
 \begin{equation*}
 M_{2}>2\max 
 \Big\{\|\overline{\varphi}\left(\kappa_1(\cdot)+\delta, \cdot\right)\|_\infty,\
 \|\overline{\varphi}\left(\kappa_3(\cdot)-\delta, \cdot\right)\|_\infty\Big\}
  \end{equation*}
with $\overline{\varphi}\left(\kappa_1+\delta, \cdot\right)$ and $\overline{\varphi}\left(\kappa_3-\delta,\cdot\right)$ being defined in $\mathbf{(1)}$  and $\mathbf{(2)}$. Thus for each $t\in[0,T]$, we have
\begin{equation}\label{condition}
\begin{array}{l}
\overline \varphi(\kappa_2(t)-\delta, t)>\overline{\varphi}\left(\kappa_1(t)+\delta,t\right)\quad \text{and }\quad\overline \varphi(\kappa_2(t)+\delta,t)>\overline{\varphi}(\kappa_3(t)-\delta,t).
 \end{array}
  \end{equation}
In view of $-\partial_xm\partial_x\overline{\varphi}\geq 0$ on $\Delta_{2\delta}$, direct calculation yields that for any $\alpha>0$,
\begin{equation*}
 \begin{array}{l}
\smallskip
 \quad \displaystyle\mathcal{L}_\alpha\overline{\varphi}-\left[\frac{1}{T}\int_0^T V\left(\kappa_1(s),s\right)\mathrm{d}s-\epsilon\right]\overline{\varphi}\\
 \smallskip
 =\displaystyle \frac{M_{2}}{\delta^2}(x-\kappa_2)\kappa'_2+\frac{M_{2}}{\delta^2}-\alpha\partial_xm\partial_x\overline{\varphi}+\left[V-{1\over T}\int_0^T V\left(\kappa_1(s),s\right)\mathrm{d}s+\epsilon\right]\overline{\varphi}\\
 \smallskip
  \geq  \displaystyle
  \frac{M_{2}}{\delta^2}-\frac{M_{2}}{\delta}\|\kappa'_2\|_\infty-M_{2}\left\|V-{1\over T}\int_0^T
  V\left(\kappa_1(s),s\right)\mathrm{d}s+\epsilon\right\|_\infty\\
  >0 \qquad \text{on }\, \Delta_{2\delta}, 
 \end{array}
\end{equation*}
where the last inequality is due to the choice of $\delta$ in \eqref{liu004}.
\smallskip

$\mathbf{(4)}$ For $(x,t)\in ((0,1)\times[0,T])\setminus (\Delta_{1\delta}\cup\Delta_{2\delta}\cup\Delta_{3\delta})$,
there is a constant $\rho>0$ independent of $(x,t)$ such that $|\partial_x m|>\rho$ and we shall construct the super-solution $\overline \varphi$ on this region by monotonically connecting the endpoints on $\partial\Delta_{i\delta}$ ($i=1,2,3$). Taking $\left\{(x,t): x\in [\kappa_1(t)+\delta,\kappa_2(t)-\delta\right],\, t\in[0,T]\}$ as an example, by \eqref{condition}, we may construct $\overline{\varphi}$ to be a positive $T$-periodic continuous function such that
$\partial_x\overline{\varphi}>0$ in this region
 and 
\begin{equation}\label{liu012}
\partial_x\overline{\varphi}((\partial \Delta_{i\delta})^+,t)
<\partial_x\overline{\varphi}((\partial \Delta_{i\delta})^-,t), \quad i=1,2.
\end{equation}
Due to $-\alpha\partial_x m\partial_x\overline{\varphi}>\alpha \rho \partial_x\overline{\varphi}$,
we can verify  that  such a function $\overline{\varphi}$ verifies
$\mathcal{L}\overline{\varphi}>0$ 
by choosing $\alpha>0$ large.
Similar constructions can be applied  to the other remaining regions. 

Now we have constructed the desired strict super-solution $\overline{\varphi}$
satisfying \eqref{ldn2} with 
 $
 \mathbb{X}= \partial\Delta_{1\delta}\cup\partial\Delta_{2\delta}\cup\partial\Delta_{3\delta} 
 $
The profile of 
$\overline{\varphi}$ defined in $\mathbf{(1)}$-$\mathbf{(4)}$ can be illustrated in Fig. \ref{figure1}. Therefore, the assertion \eqref{ldn1}  is a direct consequence of Proposition \ref{appendixprop}. 
\begin{figure}[http!!]
  \centering
\includegraphics[height=1.9in]{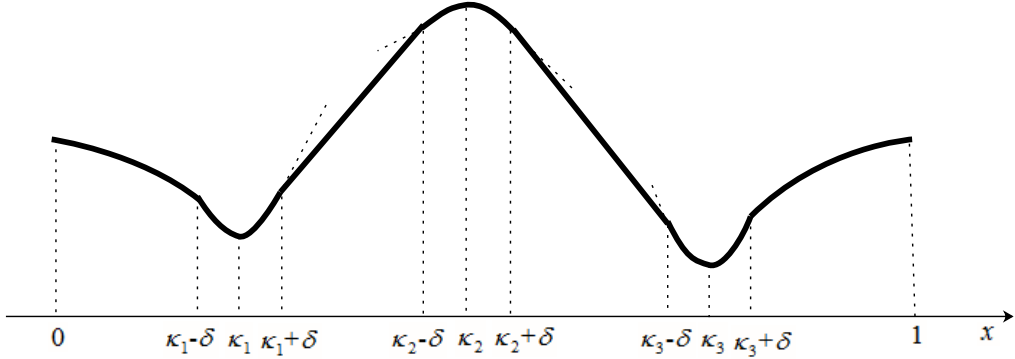}
  \caption{\small  Profile of the constructed $\overline{\varphi}$ for fixed $t$.}\label{figure1}
  \end{figure}

\smallskip
 \noindent {\bf Step 2.} We next prove
 \begin{equation}\label{ldn4}
 \limsup_{\alpha\to\infty}\lambda(\alpha)\leq\frac{1}{T}\int_0^T V\left(\kappa_1(s),s\right)\mathrm{d}s.
\end{equation}

We  employ a similar strategy as in Step 1. Given any small $\epsilon>0$, we shall construct a strict
sub-solution $\underline \varphi\geq0$ that satisfies for large $\alpha>0$,
\begin{equation}\label{liu013}
 \left\{\begin{array}{ll}
  \medskip
\mathcal{L}_\alpha\underline{\varphi} \leq\left[{1\over T}\int_0^T V\left(\kappa_1(s),s\right)\mathrm{d}s+\epsilon\right]\underline{\varphi} 
&\text{in}\,\,\left((0,1)\times [0,T]\right)\setminus\mathbb{X},\\
 \medskip
 \partial_{x}\underline{\varphi}(0,t)\geq0, \ \ \partial_{x}\underline{\varphi}(1,t)\leq0, &\text{on}\,\, [0,T], \\
 \underline{\varphi}(x,0)=\underline{\varphi}(x,T) &\text{on}\,\,(0,1), \end{array}
 \right.
\end{equation}
where $\mathbb{X}$ will be determined below so that
 $\partial_x\underline{\varphi}(x^+,t)>\partial_x\underline{\varphi}(x^-,t)$ for all $(x,t)\in\mathbb{X}$.
Once such a sub-solution $\underline{\varphi}$ exists, by Proposition \ref{appendixprop}  one can obtain \eqref{ldn4}.
For the sake of clarity, we assume $0<\kappa_1(t)<\kappa_2(t)<\kappa_3(t)< 1$ for all $t\in[0,T]$, as it is easily seen that
the following construction also work for other cases.

To this end, we  proceed as in Step 1 to find $\underline{z}\geq0$ such that
\begin{equation}\label{ldn6.1}
 \left\{\begin{array}{ll}
 \medskip
 \partial_{t}\underline{z}-\partial_{xx}\underline{z}-\alpha \partial_x m\partial_{x}\underline{z}+\left[V-V\left(\kappa_1(t),t\right)
 -\epsilon\right]\underline{z}\leq0,\,\not\equiv0\ &\text{ in }((0,1)\times[0,T])\setminus\mathbb{X},\\
 \displaystyle
 \partial_x \underline z(0,t)\geq0, \ \ \partial_x\underline z_x(1,t)\leq0  &\text{ on } [0,T],
 \end{array}
 \right.
 \end{equation}
so that
 $\underline\varphi(x,t):=\underline z(x,t)f_{1}(t)$
is a sub-solution satisfying \eqref{liu013}, where $f_{1}$ is given by \eqref{ldn7}.
For this purpose, we define $\underline z\in C([0,1]\times[0,T])$ by
\begin{equation*}
 \underline{z}(x,t):=\begin{cases}
 \smallskip
1+\frac{\epsilon\delta^2}{16}-\frac{\epsilon}{4}\left(x-\kappa_1(t)\right)^2, &x\in\left[\kappa_1(t)-\frac{\delta}{2},\kappa_1(t)+\frac{\delta}{2}\right], t\in [0,T],\\
 \smallskip
\underline{z}_1(x,t), &x\in\left(\kappa_1(t)-\delta,\kappa_1(t)-\frac{\delta}{2}\right), t\in[0,T],\\
 \smallskip
\underline{z}_2(x,t), &x\in\left(\kappa_1(t)+\frac{\delta}{2},\kappa_1(t)+\delta\right), t\in[0,T],\\
0,  & x\notin\left(\kappa_1(t)-\delta,\kappa_1(t)+\delta\right),
 \end{cases}
\end{equation*}
where $\delta>0$ is defined in Step 1 such that \eqref{liu009} holds.
For any $x\in\left[\kappa_1(t)-\frac{\delta}{2},\kappa_1(t)+\frac{\delta}{2}\right]$ and $t\in[0,T]$, noting from \eqref{liu011} that $\partial_x m\partial_{x}\underline{z}\geq0$, by choosing $\delta$ further small if necessarily we can verify  \eqref{ldn6.1} holds.
For 
any $t\in[0,T]$, we then choose $\underline{z}_1$ and $\underline{z}_2$ to fulfill the following properties:
\begin{equation*}
 \begin{cases}
 \smallskip
\partial_{x}\underline{z}_{1}(x,t)>0,\ \ \ \ \ \ x\in\left(\kappa_1(t)-\delta,\kappa_1(t)-\frac{\delta}{2}\right),\\ 
\smallskip
\underline{z}_1\left(\kappa_1(t)-\delta,t\right)=0,\ \ \ \underline{z}_1\left(\kappa_1(t)-\frac{\delta}{2},t\right)=1,\\ 
 \smallskip
 \partial_x\underline{z}_{1}\left(\left(\kappa_1-\delta/2,t\right)^-\right)
 <\frac{\epsilon\delta}{4}=\partial_x\underline{z}_{1}\left(\left(\kappa_1-\delta/2,t\right)^+\right),
 \end{cases}
 \end{equation*}
and
\begin{equation*}
 \begin{cases}
  \smallskip
\partial_{x}\underline{z}_{2}(x,t)<0,\ \ \ \ \ \ x\in\left(\kappa_1(t)+\frac{\delta}{2},\kappa_1(t)+\delta\right), \\ 
\smallskip
\underline{z}_2\left(\kappa_1(t)+\delta,t\right)=0,\ \ \ \underline{z}_2\left(\kappa_1(t)+\frac{\delta}{2},t \right)=1, \\ 
\smallskip
\partial_x\underline{z}_{1}\left(\left(\kappa_1+\delta/2\right)^+, t\right)>-\frac{\epsilon\delta}{4}=\partial_x\underline{z}_{1}\left(\left(\kappa_1+\delta/2\right)^-, t\right).
\end{cases}
\end{equation*}
 Note that  $|\partial_x m|>\rho>0$ 
 for some $\rho>0$ independent of $(x,t)$.
 Similar to $\mathbf{(4)}$ in Step 1, it can be verified readily that  $\underline z$ satisfies \eqref{ldn6.1} with
\begin{equation*}
    \begin{array}{l}
      \mathbb{X}=\Big\{(\kappa_1(t)\pm\delta,t),\ (\kappa_1(t)\pm\frac{\delta}{2},t): \, t\in[0,T] \Big\}.
     \end{array}
\end{equation*}
Thus our analysis above implies \eqref{ldn4}.

As a consequence, Proposition \ref{ldnlem1} follows from \eqref{ldn1} and \eqref{ldn4}.
\end{proof}

\begin{remark}\label{rem2.1}
{\rm {\rm (1)} Let $\kappa_1\equiv0$ and $\kappa_3\equiv1$ in \eqref{liu011}, i.e.  $m$ satisfies
\begin{equation*}
     \begin{cases}
\partial_x m(x,t)<0,\ \ &  x\in\left(0,\kappa_2(t)\right),\, t\in[0,T],\\
\partial_x m(\kappa_2(t),t)=0,\ \ &  t\in [0,T],\\
\partial_x m(x,t)>0,\ \ &  x\in\left(\kappa_2(t),1\right),\,  t\in[0,T]
\end{cases}
 \end{equation*}
 with $\kappa_2\in C^1([0,T])$ satisfying $0<\kappa_2(t)<1$ for all $t\in[0,T]$.
Then Proposition {\rm\ref{ldnlem1}} yields
$\lambda(\alpha) \to \min\left\{\frac{1}{T}\int_0^T V\left(0,s\right)\mathrm{d}s,\ \
 \frac{1}{T}\int_0^T V\left(1,s\right)\mathrm{d}s\right\}$ as $\alpha\rightarrow\infty$.

{\rm (2)} Let $\kappa_2$ and $\kappa_3$ in \eqref{liu011} 
tend to $1$,
  i.e. $m$ satisfies
 \begin{equation}\label{liu007}
     \begin{cases}
\partial_x m(x,t)>0,\ \ &  x\in\left(0,\kappa_1(t)\right),\, t\in[0,T],\\
\partial_x m(\kappa_1(t),t)=0,\ \ &  t\in [0,T],\\
\partial_x m(x,t)<0,\ \ &  x\in\left(\kappa_1(t),1\right),\,  t\in[0,T]
\end{cases}
 \end{equation}
with $\kappa_1\in C^1([0,T])$ satisfying $0\leq\kappa_1(t)\leq 1$ for $t\in[0,T]$. A illustrated example for such curve $\kappa_1$  can be shown in Fig. {\rm \ref{figure3}}, where the case $\{t\in[0,T]: \kappa(t)=0 \text{ or } \kappa(t)=1\}\neq \emptyset$ is allowed.
  By a similar but simpler argument as in Proposition {\rm \ref{ldnlem1}}, we can deduce
$\lambda(\alpha)\to\frac{1}{T}\int_0^T V\left(\kappa_1(s),s\right)\mathrm{d}s$ as $\alpha\to\infty$. In particular, if $\kappa_1\equiv0$ in \eqref{liu007}, i.e.  $\partial_x m<0$ in
$\left(0,1\right)\times[0,T]$, we can conclude that
 $\lambda(\alpha)\to \frac{1}{T}\int_0^T V\left(0,s\right)\mathrm{d}s$  as $\alpha\to\infty$, which was first proved in  \cite[Theorem {\rm 1.1}]{PZ2015}.
 }
  \begin{figure}[http!!]
  \centering
\includegraphics[height=2.1in]{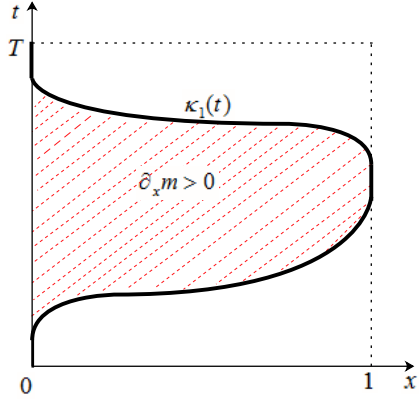}
  \caption{\small  The profile of 
  $\kappa_1$ in \eqref{liu007}, as illustrated by the black solid curve. The  red shaded region denotes the region where $\partial_x m>0$, while $\partial_x m<0$ in the  white shaded region.}\label{figure3}
  \end{figure}
\end{remark}

We are now in a position to prove Theorem \ref{ldnthm1}.
\begin{proof}[Proof of  Theorem {\rm \ref{ldnthm1}}]
Since  all spatially critical points of $m$ are nondegenerate, by the implicit function theorem, it holds that  $\kappa_i(t)\neq\kappa_j(t)$ for all $t\in[0,T]$ and $i\neq j$.
Assume $0\leq\kappa_1(t)<\ldots<\kappa_N(t)\leq1$. 
Then  we can adopt directly the arguments developed in the proof of
Propositions \ref{ldnlem1}
to  find a  nonnegative sub-solution and a positive  strict super-solution for problem \eqref{advection1}.  Finally,  we may apply Proposition \ref{appendixprop} to  conclude Theorem \ref{ldnthm1}.
\end{proof}

\section{Spatially degenerate advection: Proof of Theorem \ref{ldnthm2}}\label{Spatially_degenerate}
In this section, we are concerned with the case when the advection $m$ is spatially degenerate and prove  Theorem \ref{ldnthm2}. In particular, the set of local maximum points of $m$ allows some flat surfaces with respect to the spatial variable.
The results in some typical cases will be presented and proved first, and Theorem \ref{ldnthm2} then follows
by the similar arguments. Recall that
$\lambda^{pq}\big((\underline{\kappa},\overline{\kappa})\big)$ for $p,q\in\{\mathcal{N},\mathcal{D}\}$
denotes the principal eigenvalue of problem \eqref{definition}.

\begin{prop}\label{ldnlem2}
Suppose that there exist $T$-periodic functions $\kappa_1,\kappa_2\in C^1(\mathbb{R})$ such that $0\leq\kappa_1(t)<\kappa_2(t)\leq1$ for all $t\in[0,T]$ and
$$
 \begin{cases}
\partial_x m(x,t)>0,\ \ &  x
\in\left(0,\kappa_1(t)\right),\,\, t\in[0,T],\\
\partial_x m(x,t)=0,\ \ &  x
\in [\kappa_1(t),\,\kappa_2(t)], \,\,t\in[0,T],\\
\partial_x m(x,t)<0,\ \ &  x
\in\left(\kappa_2(t),1\right), \,\,\, t\in[0,T]. 
\end{cases}
$$
Let $\lambda(\alpha)$ be the principal eigenvalue of \eqref{advection1}. Then
$$\lim_{\alpha\rightarrow\infty}\lambda(\alpha)=\lambda^{\mathcal{N}\mathcal{N}}\big((\kappa_1,\kappa_2)\big).$$
\end{prop}
The advection $m$ given in  Proposition \ref{ldnlem2} can be illustrated in Fig. \ref{figure3_1}.
  \begin{figure}[http!!]
  \centering
\includegraphics[height=2.0in]{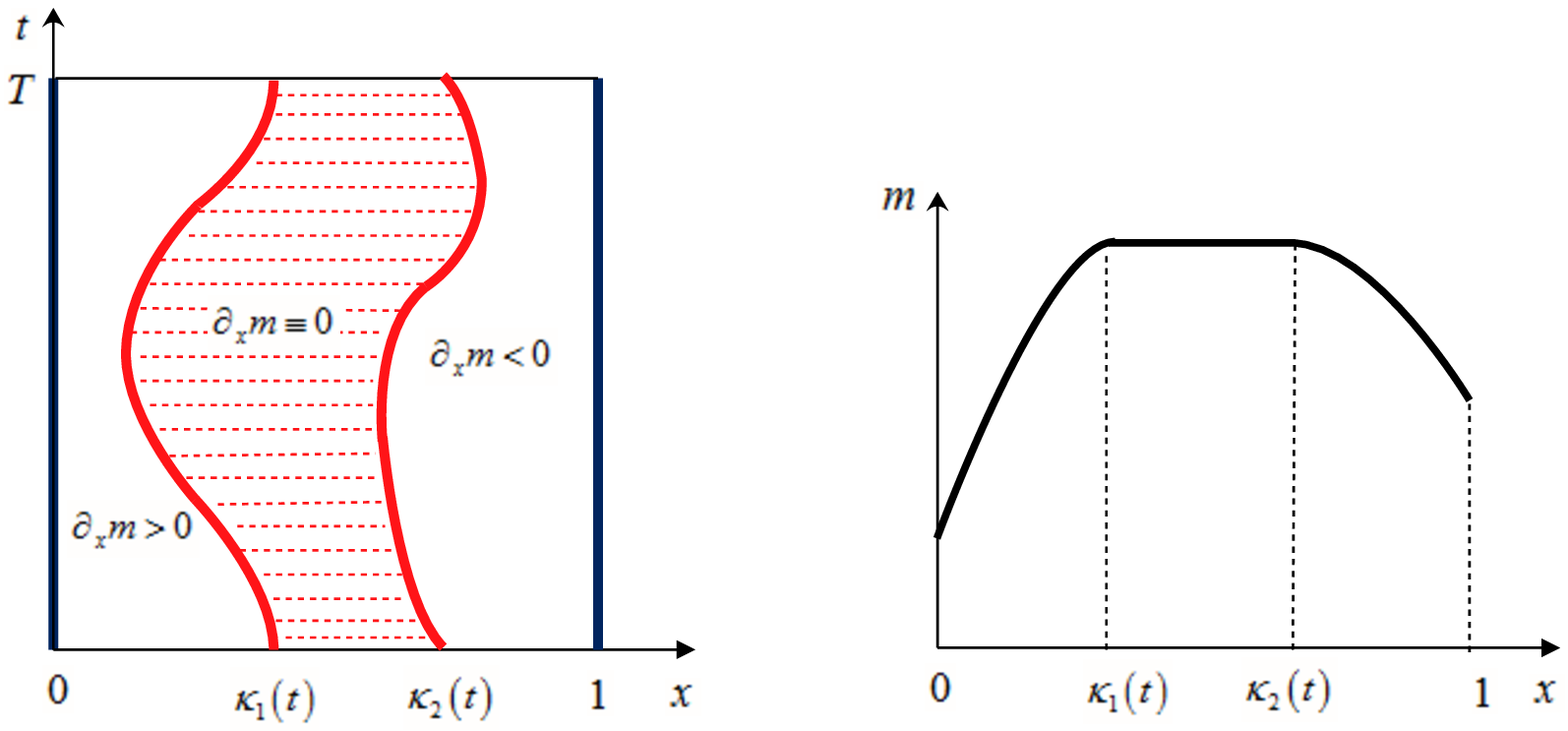}
  \caption{\small The profile of function $m$ given in Proposition \ref{ldnlem2}. In the left side picture, the red
colored curves correspond to functions $\kappa_1,\kappa_2$ and the  red shaded region denotes the region where
$\partial_x m=0$. The black colored curve in the right side picture represents the graph of $m$ for fixed $t\in[0,T]$.}\label{figure3_1}
  \end{figure}

\begin{proof}[Proof of  Proposition {\rm \ref{ldnlem2}}] Given $\eta\in\mathbb{R}$, our analysis begins with the
following auxiliary  problem:
\begin{equation}\label{liu001}
 \left\{\begin{array}{ll}
 \medskip
 \partial_{t}\psi-\partial_{xx}\psi+V(x,t)\psi=\lambda\psi,\ \ & x
\in (\kappa_1(t),\,\kappa_2(t)), \,\,t\in[0,T],\\
 \medskip
 \partial_{x}\psi\left(\kappa_1,t\right)=\eta\psi\left(\kappa_1,t\right), \partial_{x}\psi\left(\kappa_2,t\right)=-\eta\psi\left(\kappa_2,t\right), \ \ & t\in [0,T], \\
 \psi(x,0)=\psi(x,T), &x\in \left[\kappa_1(0),\kappa_2(0)\right].
 \end{array}
 \right.
 \end{equation}
Denote by $\lambda^{\mathcal{N}}_\eta$ its principal eigenvalue and $\psi^{\mathcal{N}}_\eta>0$ the corresponding principal eigenfunction. It is well known that $\lambda^{\mathcal{N}}_\eta$ is increasing and analytical with respect to $\eta\in\mathbb{R}$. Clearly,
\begin{equation}\label{liu002}
\lim_{\eta\rightarrow0}\lambda^{\mathcal{N}}_\eta=\lambda^{\mathcal{N}\mathcal{N}}\big((\kappa_1,\kappa_2)\big),
\end{equation}
and when $\eta<0$, it holds that $\partial_{x}\psi^{\mathcal{N}}_\eta\left(\kappa_1(t),t\right)<0$ and $\partial_{x}\psi^{\mathcal{N}}_\eta\left(\kappa_2(t),t\right)>0$ for all $t\in[0,T]$.

\smallskip

\noindent {\bf Step 1. } We first prove
 \begin{equation}\label{ldn5}
 \liminf_{\alpha\to\infty}\lambda(\alpha)\geq\lambda^{\mathcal{N}\mathcal{N}}\big((\kappa_1,\kappa_2)\big).
\end{equation}
Inspired from the proof of Proposition \ref{ldnlem1}, we 
shall construct a positive  super-solution $\overline \varphi$, i.e. for any given constant $\eta<0$,
we aim to find the curve set $\mathbb{X}$ and $\overline \varphi>0$  such that
 \begin{equation} \label{ldn6}
 \left\{\begin{array}{ll}
  \medskip
 \mathcal{L}_{\alpha}\overline{\varphi}\geq\lambda^{\mathcal{N}}_\eta\overline{\varphi} &{\text{in}}\,\,\left((0,1)\times[0,T]\right)\setminus\mathbb{X},\\
 \medskip
 \partial_x\overline{\varphi}(0,t)\leq0,\ \  \partial_x\overline{\varphi}(1,t)\geq0  & {\text{on}}\,\,[0,T], \\
 \medskip
 \overline{\varphi}(x,0)=\overline{\varphi}(x,T) &{\text{on}}\,\,(0,1),\\
 \partial_{x}\overline{\varphi}(x^+,t)<  \partial_{x}\overline{\varphi}(x^-,t) &\text{in}\,\,\mathbb{X}
 \end{array}
 \right.
\end{equation}
 for sufficiently large $\alpha>0$. Then \eqref{ldn5} follows from Proposition \ref{appendixprop}. For this purpose, we first consider the case $0<\kappa_1<\kappa_2<1$ and  construct $\overline \varphi$ on the following different regions.

\smallskip

 $\mathbf{(1)}$ For $ x \in \left[\kappa_1(t),\kappa_2(t)\right]$ and $t\in[0,T]$, we define $\overline \varphi(x,t):=\psi^{\mathcal{N}}_\eta(x,t)$ with $\eta<0$. In view of $\partial_x m(x,t)=0$, by the definition of $\psi^{\mathcal{N}}_\eta$ in \eqref{liu001}, one can check that such a function $\overline \varphi$ satisfies $\mathcal{L}_{\alpha}\overline{\varphi}\geq\lambda^{\mathcal{N}}_\eta\overline{\varphi}$ on this region for all $\alpha>0$.

\smallskip

 $\mathbf{(2)}$ For $x\in\left[\kappa_1(t)-\delta,\kappa_1(t)\right]$ and $t\in[0,T]$ with small $\delta>0$ to be determined later, we define
 \begin{equation*}
\begin{array}{l}
 \overline \varphi(x,t):=\psi^{\mathcal{N}}_\eta\left(\kappa_1(t),t\right)\left[1-\left(x-\kappa_1(t)\right)r-\left(x-\kappa_1(t)\right)^2R\right],
 \end{array}
\end{equation*}
where  the constants $r, R>0$ will be specified below. Due to $\eta<0$, there holds $\partial_x\psi^{\mathcal{N}}_{\eta}\left(\kappa_1(t),t\right)<0$ for all $ t\in[0,T]$. Fix  $r>0$  small such that
\begin{equation*}
\begin{array}{l}
\partial_x\overline{\varphi}(\kappa_1^-,t)
=-r\psi^{\mathcal{N}}_\eta\left(\kappa_1,t\right)
>\partial_x\psi^{\mathcal{N}}_{\eta}\left(\kappa_1,t\right)
=\partial_x\overline{\varphi}(\kappa_1^+,t),\ \,\forall t\in[0,T].
 \end{array}
\end{equation*}
 Direct calculation gives
\begin{equation*}
 \begin{array}{ll}
 \medskip
 \mathcal{L}_{\alpha}\overline{\varphi}-\lambda^{\mathcal{N}}_\eta\overline{\varphi}
  =&\hspace{-7pt}\frac{{\rm d}\psi^{\mathcal{N}}_{\eta}\left(\kappa_1,t\right)}{{\rm d} t}\left[1-\left(x-\kappa_1\right)r-\left(x-\kappa_1\right)^2R\right]+ \psi^{\mathcal{N}}_{\eta}\left(\kappa_1,t\right)\left[r \kappa'_1+2R\left(x-\kappa_1\right)\kappa'_1\right]\\
 \medskip
  &\hspace{-7pt}-\alpha \partial_x m\cdot\psi^{\mathcal{N}}_{\eta}\left(\kappa_1,t\right)\left[-r-2\left(x-\kappa_1\right)R\right]+2R\psi^{\mathcal{N}}_{\eta}\left(\kappa_1,t\right)+V(x,t)\overline{\varphi}
  -\lambda^{\mathcal{N}}_\eta\overline{\varphi}.
\end{array}
\end{equation*}
Set $R:=\frac{r}{2\delta}$. Since $\partial_x m(x,t)>0$ for
$x\in\left[\kappa_1(t)-\delta,\kappa_1(t)\right)$,  it holds that
\begin{equation*}
 \begin{array}{l}
- \partial_xm\cdot \psi^{\mathcal{N}}_{\eta}\left(\kappa_1,t\right)\left[-r
-2\left(x-\kappa_1\right)R\right]\geq0, \ \,\forall t\in[0,T],
\end{array}
\end{equation*}
whence we calculate that
\begin{equation*}
 \begin{array}{ll}
 \medskip
 \displaystyle
  \mathcal{L}_{\alpha}\overline{\varphi}-\lambda^{\mathcal{N}}_\eta\overline{\varphi}
  \geq&
  \displaystyle
  \hspace{-8pt}-\left|\frac{{\rm d}\psi^{\mathcal{N}}_{\eta}\left(\kappa_1,t\right)}{{\rm d} t}\right|
  \left(1+\frac{r\delta}{2}\right)-\frac{3r}{2}\psi^{\mathcal{N}}_{\eta}\left(\kappa_1,t\right)|\kappa'|+2R\psi^{\mathcal{N}}_{\eta}\left(\kappa_1,t\right)\\
  &
  \displaystyle
  \hspace{-8pt}-\left|V(x,t)-\lambda^{\mathcal{N}}_\eta\right|
  \psi^{\mathcal{N}}_{\eta}\left(\kappa_1,t\right)\left(1+\frac{r\delta}{2}\right).
\end{array}
\end{equation*}
We can 
choose  $\delta$ small  and thus $R$ large to obtain $ \mathcal{L}_{\alpha}\overline{\varphi}\geq\lambda^{\mathcal{N}}_\eta\overline{\varphi}$ for all $\alpha>0$. 

\smallskip

 $\mathbf{(3)}$  For $x\in\left[0,\kappa_1(t)-\delta\right]$ and $t\in[0,T]$ with $\delta>0$ defined in $\mathbf{(2)}$, one may construct  $\overline{\varphi}>0$ to be $T$-periodic and satisfy for each $t\in[0,T]$,
$$\begin{cases}
\smallskip
\overline{\varphi}\left(\kappa_1(t)-\delta, t\right)
=\psi^{\mathcal{N}}_\eta\left(\kappa_1(t),t\right)\left[1+r\delta- R \delta^2\right], \\ 
\smallskip
\partial_x\overline{\varphi}(x,t)<0, \quad x\in\left[0,\kappa_1(t)-\delta\right], \\
\partial_x\overline{\varphi}(\left(\kappa_1(t)-\delta\right)^+,t)
<\partial_x\overline{\varphi}(\left(\kappa_1(t)-\delta\right)^-,t), \\
\end{cases}$$
where the first equality is to ensure the continuity of $\overline{\varphi}$.
 Noting that $|\partial_x m|>\rho$ for all $x\in \left[0,\kappa_1(t)-\delta\right]$ with some constant $\rho>0$ independent of $(x,t)$, we may deduce
$\mathcal{L}_{\alpha}\overline{\varphi}\geq\lambda^{\mathcal{N}}_\eta\overline{\varphi}$  by choosing $\alpha>0$ large.

\smallskip

 $\mathbf{(4)}$  For $x\in\left[\kappa_2(t),1\right]$ and $ t\in[0,T]$, similar to  $\mathbf{(2)}$  and  $\mathbf{(3)}$, we can define $\overline \varphi$ as follow:
\begin{equation*}
\begin{cases}
\smallskip
\overline \varphi(x,t)=\psi^{\mathcal{N}}_\eta\left(\kappa_2,t\right)\left[1+\tilde{r}\left(x-\kappa_2\right)
+\tilde{R}\left(x-\kappa_2\right)^2\right], \,\,&x\in(\kappa_2(t),\kappa_2(t)+\delta), t\in[0,T],\\
\smallskip
\partial_x\overline{\varphi}(x,t)>0, \,\,&x\in\left[\kappa_2(t)+\delta,1\right], t\in[0,T],\\
\partial_x\overline{\varphi}(\left(\kappa_2+\delta\right)^+,t)<\partial_x\overline{\varphi}(\left(\kappa_2+\delta\right)^-, t),\
&t\in[0,T].
 \end{cases}
\end{equation*}
Proceeding as in  $\mathbf{(2)}$  and  $\mathbf{(3)}$, we can  verify, by choosing small positive constants $\delta, \tilde{r}$ and large $\tilde{R}$, that $\mathcal{L}_{\alpha}\overline{\varphi}\geq\lambda^{\mathcal{N}}_\eta\overline{\varphi}$ for large $\alpha>0$.

By summarizing the above arguments, for the case $0<\kappa_1<\kappa_2<1$, we have constructed a strict super-solution $\overline \varphi>0$ satisfying  \eqref{ldn6} 
with
\begin{equation}\label{liu014}
 \mathbb{X}=\Big\{(\kappa_1(t)-\delta,t), \, (\kappa_1(t),t), \, (\kappa_2(t),t), \, (\kappa_2(t)+\delta,t): \, t\in[0,T]\Big\}.
\end{equation}

For the remaining case $\{t\in[0,T]: \kappa_1(t)=0\}\cup \{t\in[0,T]: \kappa_2(t)=1\}\neq\emptyset$, it can be verified that the above construction also work by noting that $\partial_{x}\psi^{\mathcal{N}}_{\eta}\left(\kappa_1,t\right)<0$ and  $\partial_{x}\psi^{\mathcal{N}}_{\eta}\left(\kappa_2,t\right)>0$ for $\eta<0$.
Therefore, applying Proposition \ref{appendixprop}, we conclude that for any $\eta<0$,
$$\liminf_{\alpha\to\infty}\lambda(\alpha)\geq\lambda^{\mathcal{N}}_\eta,$$
which together with \eqref{liu002} gives the desired \eqref{ldn5} by letting $\eta \searrow 0$.
This completes Step 1.

\smallskip

\noindent {\bf Step 2. } We prove
 \begin{equation}\label{ldn8}
 \limsup_{\alpha\to\infty}\lambda(\alpha)\leq\lambda^{\mathcal{N}\mathcal{N}}\big((\kappa_1,\kappa_2)\big).
\end{equation}
For each $\eta>0$, it suffices to find a strict sub-solution $\underline\varphi\geq 0$ such that for large $\alpha>0$,
\begin{equation} \label{ldn9}
 \left\{\begin{array}{ll}
  \medskip
 \mathcal{L}_\alpha\underline{\varphi}\leq\lambda^{\mathcal{N}}_\eta\underline{\varphi} &{\text{in}}\,\,\left((0,1)\times[0,T]\right)\setminus\mathbb{X},\\
 \medskip
 \partial_{x}\underline{\varphi}(0,t)\geq0,\ \ \partial_{x}\underline{\varphi}(1,t)\leq0\  &{\text{on}}\,\,[0,T], \\
  \medskip
 \underline{\varphi}(x,0)=\underline{\varphi}(x,T) &{\text{on}}\,\,(0,1),
 \end{array}
 \right.
\end{equation}
and  $\partial_x\underline{\varphi}(x^+,t)>\partial_x\underline{\varphi}(x^-,t)$ for all $(x,t)\in\mathbb{X}$, where $\mathbb{X}$ is defined by \eqref{liu014}.

To this end,  set $\underline{\varphi}(x,t):=\psi^{\mathcal{N}}_\eta(x,t)$  for $x\in\left[\kappa_1(t),\kappa_2(t)\right]$ and $t\in[0,T]$, where $\psi^{\mathcal{N}}_\eta>0$ denotes the principal eigenfunction of \eqref{liu001} with $\eta>0$. Since $\partial_x m(x,t)=0$, it follows that
 $ \mathcal{L}_\alpha\underline{\varphi}\leq\lambda^{\mathcal{N}}_\eta\underline{\varphi}$ 
 obviously. Observe that for any $\eta>0$, 
 \begin{equation*}
\begin{array}{l}
\partial_x\psi^{\mathcal{N}}_{\eta}\left(\kappa_1,t\right)>0\quad
\mathrm{and} \quad \partial_x\psi^{\mathcal{N}}_{\eta}\left(\kappa_2,t\right)<0,\ \ \forall t\in[0,T].
\end{array}
\end{equation*}
Such  sub-solution $\underline\varphi$ satisfying \eqref{ldn9} can be constructed by the similar arguments as in Step 1. Then we may apply Proposition \ref{appendixprop} to deduce
$$\liminf_{\alpha\to\infty}\lambda(\alpha)\leq\lambda^{\mathcal{N}}_\eta \quad \text{for all }\eta>0.
$$
for which sending $\eta\searrow 0$ implies \eqref{ldn8} by using \eqref{liu002}.

A combination of \eqref{ldn5} and \eqref{ldn8} yields Proposition \ref{ldnlem2}, and the proof is complete.
\end{proof}

\begin{remark}\label{ldnrem3.2}
{\rm When the strip $(\kappa_1(t),\kappa_2(t))$ in Proposition {\rm\ref{ldnlem2}} shrinks to one curve $\kappa\in C^1(\mathbb{R})$ as illustrated in Fig. {\rm\ref{figure4}}, we can assert that  Proposition {\rm \ref{ldnlem2}} coincides with 
the result in Remark {\rm\ref{rem2.1}{\rm (2)}} by the following observation:
\begin{equation}\label{ldnaux}
    \lim_{\delta\rightarrow0^+}\lambda^{\mathcal{N}\mathcal{N}}\big((\kappa-\delta,\kappa+\delta)\big)=\frac{1}{T}\int_0^T V\left(\kappa(s),s\right)\mathrm{d}s.
\end{equation}
 The proof of \eqref{ldnaux} is given in Appendix {\rm \ref{appenC}} for the sake of completeness.
}

\begin{figure}[http!!]
  \centering
\includegraphics[height=1.5in]{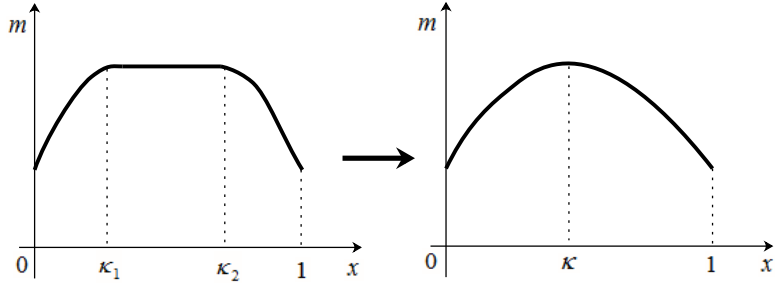}
  \caption{\small  The profile of advection $m$ in Proposition {\rm \ref{ldnlem2}} for fixed $t\in[0,T]$ as the two curve $\kappa_1$ and $\kappa_2$  shrink to $\kappa\in C(\mathbb{R})$. }\label{figure4}
  \end{figure}
\end{remark}

\begin{prop}\label{ldnlem3}
Suppose that there exist $T$-periodic functions $\kappa_1,\kappa_2\in C^1(\mathbb{R})$ such that $0\leq\kappa_1(t)<\kappa_2(t)\leq1$ for all $t\in[0,T]$ and
 $$\begin{cases}
\partial_x m(x,t)<0,\ \ & x \in\left(0,\kappa_1(t)\right), \, t\in[0,T],\\
\partial_x m(x,t)=0,& x \in\left[\kappa_1(t),\kappa_2(t)\right], \, t\in[0,T],\\
\partial_x m(x,t)>0,& x\in\left(\kappa_2(t),1\right), \, t\in[0,T].
\end{cases}$$
 \begin{itemize}
    \item [{\text(i)}] If $0<\kappa_1(t)<\kappa_2(t)<1$ for all $t\in[0,T]$,  then
 $$
 \lim_{\alpha\rightarrow\infty}\lambda(\alpha)
 =\min\left\{\lambda^{\mathcal{D}\mathcal{D}}\big((\kappa_1,\kappa_2)\big),\ \,\frac{1}{T}
 \int_0^T V\left(0,s\right)\mathrm{d}s,\ \ \frac{1}{T}\int_0^T V\left(1,s\right)\mathrm{d}s\right\};
 $$

    \item [{\text(ii)}] If $\kappa_1\equiv0$ and $0<\kappa_2(t)<1$ for all $t\in[0,T]$, then
 $$
  \lim_{\alpha\rightarrow\infty}\lambda(\alpha)
  =\min\left\{\lambda^{\mathcal{N}\mathcal{D}}\big((0,\kappa_2)\big),\ \
  \frac{1}{T}\int_0^T V\left(1,s\right)\mathrm{d}s\right\};
  $$

  \item [{\text(iii)}] If  $\kappa_2\equiv1$  and $0<\kappa_1(t)<1$ for all $t\in[0,T]$, then
   $$
   \lim_{\alpha\rightarrow\infty}\lambda(\alpha)
   =\min\left\{\lambda^{\mathcal{D}\mathcal{N}}\big((\kappa_1,1)\big),\ \
   \frac{1}{T}\int_0^T V\left(0,s\right)\mathrm{d}s\right\}.
   $$
 \end{itemize}
\end{prop}
The advection $m$ given in  Proposition \ref{ldnlem3}{\rm (i)} can be illustrated in Fig. \ref{figure3_2}.
  \begin{figure}[http!!]
  \centering
\includegraphics[height=2.0in]{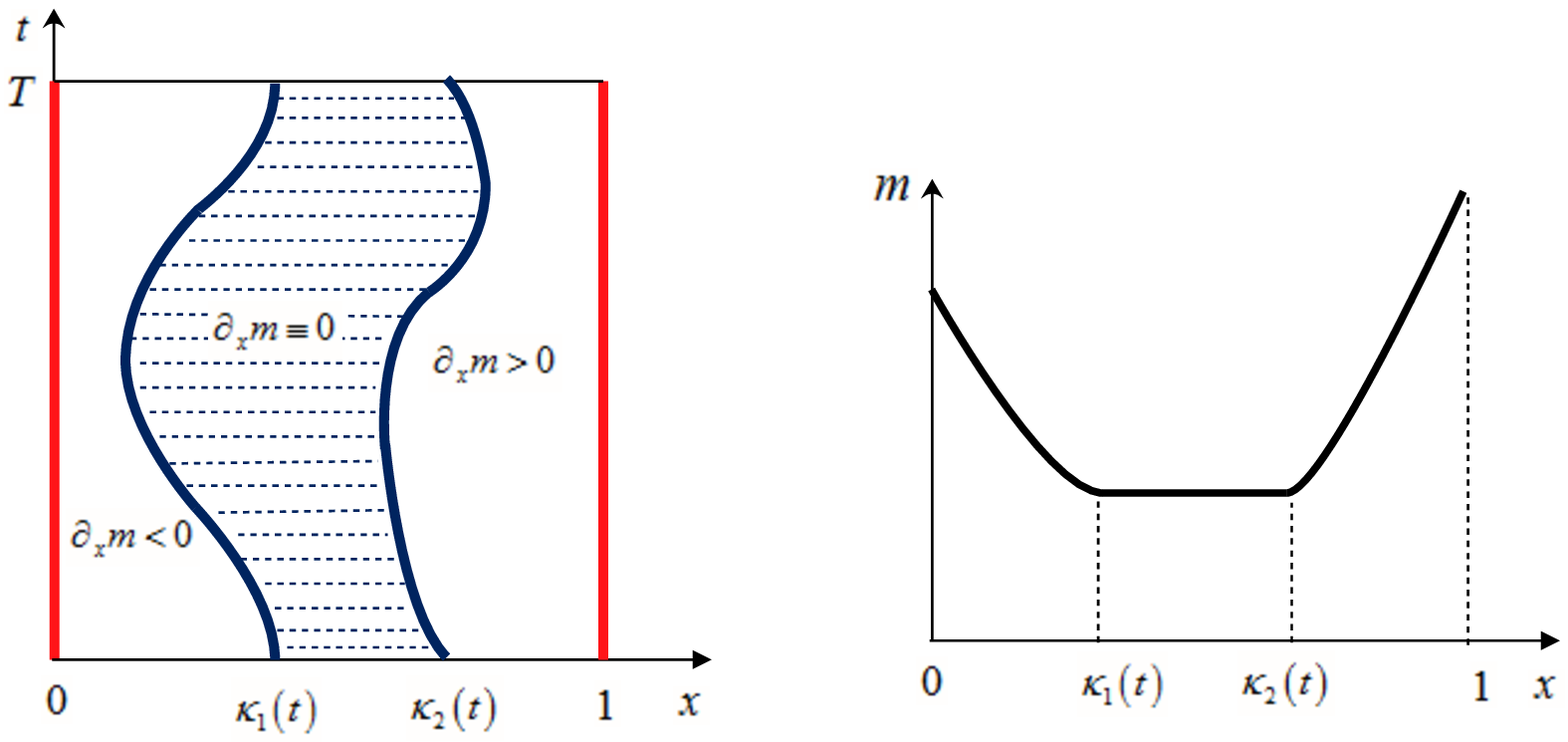}
  \caption{\small The profile of function $m$ given in Proposition \ref{ldnlem2}{\rm (i)}. In the left side picture, the blue colored curve correspond to functions $\kappa_1,\kappa_2$ and the blue shaded region denotes the region where
$\partial_x m=0$. The black colored curve in the right side picture represents the graph of $m$ for fixed $t\in[0,T]$.}\label{figure3_2}
  \end{figure}

\begin{proof}[Proof of Proposition {\rm \ref{ldnlem3}}]
We only prove {\rm(i)} since the assertions {\rm(ii)} and {\rm(iii)} follow from a similar but simpler argument.
Given any $\eta\in\mathbb{R}$, consider the eigenvalue problem
\begin{equation}\label{liu003}
 \left\{\begin{array}{ll}
 \medskip
 \partial_{t}\psi-\partial_{xx}\psi+V(x,t)\psi=\lambda\psi,\ \ &x\in\left(\kappa_1(t),\kappa_2(t)\right), t\in[0,T],\\
 \medskip
 \psi\left(\kappa_1,t\right)=\eta\partial_{x}\psi\left(\kappa_1,t\right),\,  \psi\left(\kappa_2,t\right)=-\eta\partial_{x}\psi\left(\kappa_2,t\right), \ \ & t\in[0,T], \\
 \psi(x,0)=\psi(x,T) &x\in [\kappa_1(0),\kappa_2(0)].
 \end{array}
 \right.
 \end{equation}
Let $\lambda^{\mathcal{D}}_\eta$ denote the principal eigenvalue of \eqref{liu003} and $\psi^{\mathcal{D}}_\eta>0$ be the corresponding principal eigenfunction.
For each $\eta\neq 0$, it is easily seen that $\psi^{\mathcal{D}}_\eta\left(\kappa_1(t), t\right)>0$ and $\psi^{\mathcal{D}}_\eta\left(\kappa_2(t), t\right)>0$ for all $ t\in[0,T]$, and
 $\lambda^{\mathcal{D}}_\eta\to\lambda^{\mathcal{D}\mathcal{D}}\big((\kappa_1,\kappa_2)\big)$ as $\eta\searrow 0$.
Given any $\eta\in\mathbb{R}$, set
$$\lambda_{\eta}^{\min}:=\min\left\{\lambda^{\mathcal{D}}_\eta,\ \,
\frac{1}{T}\int_0^T V\left(0,s\right)\mathrm{d}s,\ \, \frac{1}{T}\int_0^T V\left(1,t\right)\mathrm{d}t\right\}.
$$
Then it suffices to show $\lambda(\alpha)\to \lambda_{0}^{\min}$ as $\alpha\rightarrow\infty$.

\smallskip

\noindent {\bf Step 1. } We first establish
 \begin{equation}\label{ldn13}
 \liminf_{\alpha\to\infty}\lambda(\alpha)\geq\lambda_{0}^{\min}.
\end{equation}
Given any $\epsilon>0$ and $\eta<0$,  we  shall construct a strict super-solution $\overline \varphi>0$ such that 
 \begin{equation} \label{ldn10}
 \left\{\begin{array}{ll}
  \medskip
\mathcal{L}_{\alpha}\overline{\varphi}\geq\left(\lambda^{\min}_\eta-\epsilon\right)\overline{\varphi} &{\text{in}}\,\,\left((0,1)\times[0,T]\right)\setminus\mathbb{X},\\
 \medskip
 \partial_{x}\overline{\varphi}(0,t)\leq0, \ \, \partial_{x}\overline{\varphi}(1,t)\geq0  & {\text{in}}\,\,[0,T], \\
 \overline{\varphi}(x,0)=\overline{\varphi}(x,T) & \text{in}\,\,(0,1),
 \end{array}
 \right.
\end{equation}
provided that  $\alpha$ is sufficiently large, where the curve set $\mathbb{X}$ will be chosen such that $ \partial_{x}\overline{\varphi}(x^+,t)<  \partial_{x}\overline{\varphi}(x^-,t)$ for all $(x,t)\in \mathbb{X}$.

\smallskip

$\mathbf{(1)}$ For $ x\in \left[\kappa_1(t),\kappa_2(t)\right]$ and $t\in[0,T]$, we set $\overline \varphi(x,t):=\psi^{\mathcal{D}}_\eta(x,t)$. Due to $\partial_x m(x,t)=0$, by the definition of $\psi^{\mathcal{D}}_\eta$ in \eqref{liu003},  we derive that
$$\mathcal{L}_{\alpha}\overline{\varphi}-\left(\lambda^{\min}_\eta-\epsilon\right)\overline{\varphi}\geq\partial_t\psi^{\mathcal{D}}_\eta-\partial_{xx}\psi^{\mathcal{D}}_\eta+V(x,t)\psi^{\mathcal{D}}_\eta-\lambda^{\mathcal{D}}_\eta\psi^{\mathcal{D}}_\eta=0.$$

\smallskip

$\mathbf{(2)}$ For $x\in\left[\kappa_1(t)-\delta_1,\kappa_1(t)\right]$ and $t\in[0,T]$, we choose
 \begin{equation*}
 \overline \varphi(x,t):=\psi^{\mathcal{D}}_\eta\left(\kappa_1(t),t\right)\left[1-\tfrac{\left(x-\kappa_1(t)\right)^2}{2\delta_1}\right],
\end{equation*}
where the small constant $\delta_1>0$  will be selected below.
In view of $\eta<0$, it follows that
\begin{equation*}
 \partial_{x}\overline{\varphi}(\kappa_1^-,t)
 =0>\partial_x\psi^{\mathcal{N}}_{\eta}\left(\kappa_1,t\right)=\partial_{x}\overline{\varphi}(\kappa_1^+,t),\ \ \forall t\in[0,T].
\end{equation*}
Due to $- \alpha \partial_x m\cdot\partial_{x}\overline{\varphi}\geq0$ in this region, we can choose $\delta_1$
 small such  that for any $\alpha>0$,
 \begin{equation*}
 \begin{array}{l}
   \medskip
 ~~~~~~~\,\,\,\, \mathcal{L}_{\alpha}\overline{\varphi}-\left(\lambda^{\min}_\eta-\epsilon\right)\overline{\varphi}\\
  \medskip
  \geq \frac{{\rm d}\psi^{\mathcal{D}}_\eta\left(\kappa_1,t\right)}{{\rm d} t}\left[1-\frac{\left(x-\kappa_1\right)^2}{2\delta_1}
  \right]+\frac{(x-\kappa_1)\kappa'_1}{\delta_1}\psi^{\mathcal{D}}_\eta
  \left(\kappa_1,t\right)+\frac{\psi^{\mathcal{D}}_\eta
  \left(\kappa_1,t\right)}{\delta_1}\\
   \medskip
  \quad +(V-\lambda^{\min}_\eta)\psi^{\mathcal{D}}_\eta\left(\kappa_1,t\right)
  \left[1-\frac{\left(x-\kappa_1\right)^2}{2\delta_1}\right]\\
   \medskip
 \geq -\left|\frac{{\rm d}\psi^{\mathcal{D}}_\eta\left(\kappa_1,t\right)}{{\rm d} t}\right|-|\kappa'_1|\psi^{\mathcal{D}}_\eta
  \left(\kappa_1,t\right)
 -\left|V-\lambda^{\min}_\eta\right|\psi^{\mathcal{D}}_\eta\left(\kappa_1,t\right)
 +\frac{\psi^{\mathcal{D}}_\eta\left(\kappa_1,t\right)}{\delta_1}\\
  \geq0,  \qquad \forall (x,t)\in\ \left[\kappa_1-\delta_1,\kappa_1\right]\times[0,T].
\end{array}
\end{equation*}

\smallskip

$\mathbf{(3)}$ We first fix some constant  $0<\delta_2\ll 1$ such that
\begin{equation}\label{liu22}
|V(x,t)-V\left(0,t\right)|<\epsilon/2\,\,\text{ on }\,\left[0,\delta_2\right]\times[0,T].
\end{equation}
For $(x,t)\in\left[0,\delta_2\right]\times[0,T]$, one may construct
$\overline \varphi$ of variable-separated form:
  $$\overline{\varphi}(x,t):=\overline{z}_0(x)f_0(t),$$
where $f_0$  is a $T$-periodic function given by
 \begin{equation}\label{def_f_0}
 \begin{array}{ll}
 \displaystyle
 f_0(t)=\exp\left[{-\int_0^tV\left(0,s\right)\mathrm{d}s+{t\over T}\int_0^T V
 \left(0,s\right)\mathrm{d}s}\right].
 \end{array}
\end{equation}
Set $\overline{z}_0(x):=\varepsilon\left(M+x^2\right)$,
where $M>\frac{4}{\epsilon}$ is given and  $\varepsilon>0$ is chosen to satisfy
\begin{equation}\label{liu015}
\begin{array}{l}
\overline{\varphi}(\delta_2,t)=\varepsilon\left(M+\delta_2^2\right)f_0(t)
<\frac{1}{2}\psi^{\mathcal{D}}_\eta\left(\kappa_1,t\right)
=\overline{\varphi}\left(\kappa_1-\delta_1,t\right).
\end{array}
\end{equation}
Note that $-\alpha \partial_x m\partial_{x}\overline{z}_0
=-2\varepsilon\alpha x\partial_x m\geq0$ on $\left[0,\delta_2\right]\times[0,T]$.
By  the choice of $\delta_2$ and $M$, we have 
 \begin{equation*}
 \begin{array}{ll}
 -\partial_{xx}\overline{z}_0-\alpha \partial_x m
 \partial_{x}\overline{z}_0+\left[V(x,t)-V\left(0,t\right)
 +\epsilon\right]\overline{z}_0>0. 
 \end{array}
  \end{equation*}
We therefore calculate that
\begin{equation*}
 \begin{array}{l}
   \medskip
 ~~~~~~~\,\,\,\, \mathcal{L}_{\alpha}\overline{\varphi}-\left(\lambda^{\min}_\eta-\epsilon\right)\overline{\varphi}\\
  \medskip
  \displaystyle
  \geq\partial_{t}\overline{\varphi}-\partial_{xx}\overline{\varphi}- \alpha \partial_x m\partial_{x}\overline{\varphi}+V\overline{\varphi}-\left[{1\over T}\int_0^T V\left(0,s\right)\mathrm{d}s-\epsilon\right]\overline{\varphi}\\
   \medskip
   \displaystyle
 =\Big[-\partial_{xx}\overline{z}_0-\alpha \partial_x m\partial_{x}\overline{z}_0+\left(V-V\left(0,t\right)
 +\epsilon\right)\overline{z}_0\Big]f_0(t)\\
  >0, \quad \forall (x,t)\in\left[0,\delta_2\right]\times[0,T].
\end{array}
\end{equation*}
Moreover, noting that $\partial_{x}\overline{\varphi}(0,t)=0$ for all $ t\in[0,T]$,  the boundary conditions in \eqref{ldn10} hold.

\smallskip

$\mathbf{(4)}$ For $x\in\left[\delta_2,\kappa_1(t)-\delta_1\right]$ and $t\in[0,T]$, in view of
 $\overline{\varphi}(\delta_2,t)<\overline{\varphi}
 \left(\kappa_1-\delta_1, t\right)\, (\forall t\in[0,T])$ as given in  \eqref{liu015},
we may choose the  $T$-periodic $\overline{\varphi}>0$ such that $\partial_{x}\overline{\varphi}(x,t)>0$ and 
\begin{equation*}
\overline{\varphi}(x^+,t)=\overline{\varphi}(x^-,t)\quad\text{and}\quad \partial_{x}\overline{\varphi}\left(x^+,t\right)<\partial_{x}\overline{\varphi}\left(x^-,t\right) \quad \text{for }\, \,x=\delta_2, \, \kappa_1-\delta_1.
\end{equation*}
Since $-\partial_x m>\rho$ for some constant $\rho>0$ in dependent of $(x,t)$,
the above chosen  $\overline \varphi$ satisfies
$\mathcal{L}_{\alpha}\overline{\varphi}\geq
\left(\lambda^{\min}_\eta-\epsilon\right)\overline{\varphi}$ 
 by letting  $\alpha>0$ be large enough.

\smallskip

$\mathbf{(5)}$ For $x\in\left[\kappa_2(t),1\right]$ and $t\in[0,T]$,
similar to $\mathbf{(2)}$-$\mathbf{(4)}$, we define $\overline \varphi$ as follows:
\begin{equation*}
 \left\{
 \begin{array}{ll}
 \medskip
\overline \varphi(x,t)=\psi^{\mathcal{D}}_\eta\left(\kappa_2(t), t\right)
\left[1+\frac{1}{2\delta_3}\left(x-\kappa_2(t)\right)^2\right],\ \ &
x\in\left(\kappa_2(t),\kappa_2(t)+\delta_3\right),\, t\in[0,T],\\
 \medskip
\overline \varphi(x,t)=\tilde{\varepsilon}\left[\tilde{M}+\left(1-x\right)^2\right]f_1(t),&
x\in\left[1-\delta_4,1\right],\, t\in[0,T],\\
 \medskip
\partial_{x}\overline{\varphi}(x,t)<0, &
x\in\left(\kappa_2(t)+\delta_3, 1-\delta_4\right),\, t\in[0,T],\\
 \medskip
\partial_{x}\overline{\varphi}\left(x^+,t\right)
<\partial_{x}\overline{\varphi}(x^-,t),\ \ & x=\kappa_2(t)+\delta_3, 1-\delta_4,\, t\in[0,T],
 \end{array}
 \right.
\end{equation*}
where the $T$-periodic function $f_1$ is given by
\begin{equation}\label{def_f_11}
    \begin{array}{c} \displaystyle
      f_1(t)= \exp\left[{-\int_0^tV\left(1,s\right)\mathrm{d}s+{t\over T}\int_0^T V\left(1,s\right)\mathrm{d}s}\right].
     \end{array}
\end{equation}
Proceeding as  above, one can verify that $\overline \varphi$ satisfies $\mathcal{L}_{\alpha}\overline{\varphi}\geq\left(\lambda^{\min}_\eta-\epsilon\right)\overline{\varphi}$ by choosing small positive constants $\delta_3,\delta_4,\tilde{\varepsilon}$ and large $\tilde{M}$, provided that $\alpha$ is large enough.

Until now, we have constructed a continuous super-solution $\overline \varphi>0$ satisfying  \eqref{ldn10} with
\begin{equation*}
    \begin{array}{c}
       \mathbb{X}=\Big\{(\delta_2,t), \, (\kappa_1(t)-\delta_1,t), \, (\kappa_1(t),t), \,  (\kappa_2(t),t), \,  (\kappa_2(t)+\delta_3, t), \, (1-\delta_4,t): \, t\in [0,T] \Big\}. 
     \end{array}
 \end{equation*}
 This together with Proposition \ref{appendixprop} implies for any $\eta<0$ and $\epsilon>0$,
$$\liminf_{\alpha\to\infty}\lambda(\alpha)\geq\lambda^{\min}_\eta-\epsilon.$$
Letting $\epsilon\searrow0$ and $\eta\searrow0$ give \eqref{ldn13}. Step 1 is thus complete.

\smallskip

\noindent {\bf Step 2. } We now turn to prove
 \begin{equation}\label{ldn12}
 \limsup_{\alpha\to\infty}\lambda(\alpha)\leq\lambda_{0}^{\min}.
\end{equation}

First, we show
 $\limsup\limits_{\alpha\to\infty}\lambda(\alpha)\leq\lambda^{\mathcal{D}}_0$. Define
 \begin{equation*}
 \underline \varphi(x,t):=
 \left\{
 \begin{array}{ll}
 \medskip
\psi^{\mathcal{D}}_0(x,t), & x\in\left[\kappa_1(t),\kappa_2(t)\right],\, t\in[0,T],\\
\smallskip
0, &x\in \left[0,\kappa_1(t)\right)\cup\left(\kappa_2(t),1\right], \, t\in[0,T],
 \end{array}
 \right.
\end{equation*}
where $\psi^{\mathcal{D}}_0$ denotes the principal eigenfunction of  \eqref{liu003} with $\eta=0$.
Clearly, such a function $\underline \varphi$ is continuous and satisfies
\begin{equation} \label{ldn11}
 \left\{\begin{array}{ll}
  \medskip
\mathcal{L}_{\alpha}\underline{\varphi}\leq\lambda^{\mathcal{D}}_0\underline{\varphi}, &x\in \left((0,1)\setminus\{\kappa_1(t), \kappa_2(t)\}\right),\, t\in[0,T],\\
 \medskip
 \partial_{x}\underline{\varphi}(0,t)=0, \ \ \partial_{x}\underline{\varphi}(1,t)=0,  &t\in[0,T], \\
  \medskip
 \underline{\varphi}(x,0)=\underline{\varphi}(x,T), &x\in (0,1).
 \end{array}
 \right.
\end{equation}
We further note that, for any $t\in[0,T]$,
\begin{align*}
    &\partial_x\underline{\varphi}(\kappa_1^+,t)=\partial_x\psi^{\mathcal{D}}_0(\kappa_1,t)>0
 =\partial_x\underline{\varphi}(\kappa_1^-,t)\ \text{ and }\\
 &\partial_x\underline{\varphi}(\kappa_2^+,t)=0>\partial_x\psi^{\mathcal{D}}_0(\kappa_2,t)
 =\partial_x\underline{\varphi}(\kappa_2^-,t).
\end{align*}
Applying Proposition \ref{appendixprop} with $\mathbb{X}=\{(\kappa_1(t),t),\ (\kappa_2(t),t): t\in[0, T]\}$
and $\mathbb{T}=\emptyset$ to \eqref{ldn11}, we can obtain the desired result.

We next prove $\limsup\limits_{\alpha\to\infty}\lambda(\alpha)
\leq\frac{1}{T}\int_0^T V\left(0,s\right)\mathrm{d}s$. Choose
$\underline\varphi(x,t):=\underline z(x)f_0(t)$ with $T$-periodic function $f_0$ defined by \eqref{def_f_0}.
For any fixed $\epsilon>0$, we define
\begin{equation*}
 \underline{z}(x):=
 \left\{\begin{array}{ll}
 \smallskip
1+\frac{\epsilon\delta^2_2}{16}-\frac{\epsilon}{4}x^2,\ \ &x\in\left[0,\frac{\delta_2}{2}\right],\\
\smallskip
\underline{z}_1(x), &x\in\left(\frac{\delta_2}{2},\delta_2\right],\\
\smallskip
0,  &x\in\left(\delta_2,1\right],
 \end{array}
 \right.
\end{equation*}
where $\delta_2$ is defined by \eqref{liu22} in Step 1, and
$\underline{z}_1(x)$ satisfies 
\begin{equation*}
 \left\{\begin{array}{l}
 \smallskip
\underline{z}'_{1}<0~\ \ \,{\text{in}}\,\,\left(\frac{\delta_2}{2},\delta_2\right],\\
\smallskip
\underline{z}_1\left(\frac{\delta_2}{2}\right)=1, ~\,~\underline{z}_1\left(\delta_2\right)=0,\\
\underline{z}'_{1}\left(\frac{\delta_2}{2}^-\right)=-\frac{\epsilon\delta_2}{4}
<\underline{z}'_{1}\left(\frac{\delta_2}{2}^+\right).
 \end{array}
 \right.
\end{equation*}
Clearly, $\underline z'(\delta_2^-)<0=\underline z'(\delta_2^+)$.
Furthermore, for sufficiently large $\alpha>0$, there holds
\begin{equation*}
 \left\{\begin{array}{ll}
 \medskip
 -\underline{z}''-\alpha \partial_x m \cdot\underline{z}'+\left[V(x,t)-V\left(0,t\right)
 -\epsilon\right]\underline{z}<0 &{\text{in}}\,\,\,\,(0,1),\\
 \underline z' (0)=0, \ \ \underline z'(1)=0.
 \end{array}
 \right.
 \end{equation*}
Hence, such a function $\underline{\varphi}=\underline z(x)f_0(t)$ is a strict sub-solution in the sense that it satisfies
 \begin{equation*}
 \left\{\begin{array}{ll}
 \medskip
 \mathcal{L}_{\alpha}\underline{\varphi}\leq\left[{1\over T}\int_0^T V\left(0,s\right)\mathrm{d}s+\epsilon\right]\underline{\varphi} &{\text{in}}\,\,\left((0,1)\times[0,T]\right)\setminus \mathbb{X},\\
 \medskip
 \partial_{x}\underline{\varphi}(0,t)=0,\ \  \partial_{x}\underline{\varphi}(1,t)=0  &{\text{on}}\,\,[0,T], \\
 \medskip
 \underline{\varphi}(x,0)=\underline{\varphi}(x,T) &{\text{on}}\,\,(0,1),\\
 \partial_{x}\underline{\varphi}(x^+,t)> \partial_{x}\underline{\varphi}(x^-,t) &\text{in}\,\,\mathbb{X},
 \end{array}
 \right.
 \end{equation*}
with $\mathbb{X}=\left\{\frac{\delta_2}{2},\delta_2\right\}\times[0,T]$. By Proposition \ref{appendixprop} again, we have
 \begin{equation*}
\limsup\limits_{\alpha\to\infty}\lambda(\alpha)\leq\frac{1}{T}\int_0^T V\left(0,s\right)\mathrm{d}s.
\end{equation*}

A similar argument yields  $\limsup\limits_{\alpha\to\infty}\lambda(\alpha)\leq\frac{1}{T}\int_0^T V\left(1,t\right)\mathrm{d}t$. Hence, \eqref{ldn12} holds. Proposition \ref{ldnlem3} follows from \eqref{ldn13}) and \eqref{ldn12}.
\end{proof}

By the similar arguments as in  Propositions \ref{ldnlem2} and \ref{ldnlem3}, we can state the following results:

\begin{prop}\label{ldnlem4}
Assume that there exist $T$-periodic functions $\kappa_1,\kappa_2\in C^1(\mathbb{R})$ such that $0\leq\kappa_1(t)<\kappa_2(t)\leq1$  for all $t\in[0,T]$ and
 $$\begin{cases}
\partial_x m(x,t)>0,\ \ & x \in\left(0,\kappa_1(t)\right), \,\, t\in[0,T],\\
\partial_x m(x,t)=0,& x \in\left[\kappa_1(t),\kappa_2(t)\right], \,\, t\in[0,T],\\
\partial_x m(x,t)>0,& x\in\left(\kappa_2(t),1\right), \,\, t\in[0,T].
\end{cases}$$

 \begin{itemize}
    \item [{\text(i)}] If  $0\leq\kappa_1(t)<\kappa_2(t)<1$ for all $t\in[0,T]$, then
 $$
 \lim_{\alpha\rightarrow\infty}\lambda(\alpha)
 =\min\left\{\lambda^{\mathcal{N}\mathcal{D}}\big((\kappa_1,\kappa_2)\big),\ \
 \frac{1}{T}\int_0^T V\left(1,s\right)\mathrm{d}s\right\};
 $$

    \item [{\text(ii)}] If $0<\kappa_1(t)<1$ and $\kappa_2\equiv 1$ for all $t\in[0,T]$, then
 $$
  \lim_{\alpha\rightarrow\infty}\lambda(\alpha)
  =\lambda^{\mathcal{N}\mathcal{N}}\big((\kappa_1,1)\big).
  $$
 \end{itemize}
\end{prop}

\begin{prop}\label{ldnlem5} Assume that there exist $T$-periodic functions $\kappa_1,\kappa_2\in C^1(\mathbb{R})$ such that $0\leq\kappa_1(t)<\kappa_2(t)\leq1$  for all $t\in[0,T]$ and
 $$\begin{cases}
\partial_x m(x,t)<0,\ \ & x \in\left(0,\kappa_1(t)\right), \,\, t\in[0,T],\\
\partial_x m(x,t)=0,& x \in\left[\kappa_1(t),\kappa_2(t)\right], \,\, t\in[0,T],\\
\partial_x m(x,t)<0,& x\in\left(\kappa_2(t),1\right), \,\, t\in[0,T].
\end{cases}$$
 \begin{itemize}
    \item [{\text(i)}] If $0<\kappa_1(t)<\kappa_2(t)\leq1$ for all $t\in[0,T]$, then
$$\lim_{\alpha\rightarrow\infty}\lambda(\alpha)
=\min\left\{\lambda^{\mathcal{D}\mathcal{N}}\big((\kappa_1,\kappa_2)\big),\ \
 \frac{1}{T}\int_0^T V\left(0,s\right)\mathrm{d}s\right\};$$

\item [{\text(ii)}] If $\kappa_1\equiv 0$ and $0<\kappa_2(t)<1$ for all $t\in[0,T]$, then
 $$
  \lim_{\alpha\rightarrow\infty}\lambda(\alpha)
  =\lambda^{\mathcal{N}\mathcal{N}}\big((0,\kappa_2)\big).
  $$
 \end{itemize}
\end{prop}



\begin{proof}[Proof of Theorem {\rm \ref{ldnthm2}}]
Theorem \ref{ldnthm2} can be  established by constructing the suitable super/sub-solutions and
applying Proposition \ref{appendixprop} as before.

For each  $i=\mathbf{B}$ ($0\leq i\leq N$), define  $K_i:=\{(x,t): x\in [\kappa_i(t)-\delta, \kappa_{i+1}(t)+\delta], \, t\in [0,T]\}$ with some small $\delta>0$ to be determined later.  On the region $K_i$, we may construct the desired super-solution and sub-solution by using directly the arguments  in  Propositions \ref{ldnlem2} and \ref{ldnlem3}. Indeed, note that $\mathbf{B}=\mathbf{E}(\mathcal{N},\mathcal{N})\cup \mathbf{E}(\mathcal{N},\mathcal{D})\cup\mathbf{E}(\mathcal{D},\mathcal{N})\cup\mathbf{E}(\mathcal{D},\mathcal{D})$ as defined in Theorem \ref{ldnthm2}. When $i\in \mathbf{E}(\mathcal{N},\mathcal{N})$, the constructions  can follow those in Proposition  \ref{ldnlem2}, while we can apply the arguments in Proposition \ref{ldnlem3} for $i\in \mathbf{E}(\mathcal{D},\mathcal{D})$. Furthermore, as in Propositions \ref{ldnlem4} and \ref{ldnlem5},  the construction for $i\in \mathbf{E}(\mathcal{D},\mathcal{N})\cup \mathbf{E}(\mathcal{N},\mathcal{D})$ can be completed by integrating the ideas in Propositions  \ref{ldnlem2} and \ref{ldnlem3}. Finally,  on the remaining region $([0,1]\times[0,T])\setminus  (\bigcup_{i\in\mathbf{B}}K_i)$,  the super-solution and sub-solution can be constructed by a similar argument as in Proposition \ref{ldnlem1}.
Therefore, by choosing $\delta$ small  and applying Proposition \ref{appendixprop},  Theorem \ref{ldnthm2} can be proved.
\end{proof}

\section{Temporally degenerate advection: Proof of Theorem \ref{LTD_thm main}}\label{S4}
In this section, we consider the case when the advection $m$ allows temporal degeneracy and 
prove Theorem \ref{LTD_thm main} by examining several typical examples.
Since $\partial_xm(x,t)=b(t)$ in problem \eqref{LTD_eq 1}, the time-periodic parabolic
operator $\mathcal{L}_{\alpha}$  now becomes
$$\mathcal{L}_{\alpha}=\partial_{t}-\partial_{xx}- \alpha b(t)\partial_x+V.$$
We begin with the following result.

\begin{prop}\label{LTD_thm 1} Assume that the $T$-periodic function $b\in C(\mathbb{R})$ satisfies
$$ b(t)\equiv0,\ \ \forall t\in\left[0,t_*\right]\quad \text{and} \quad b(t)>0,\ \
\forall t\in\left(t_*,T\right)\ \, {\text for\  some}\ t_*\in(0,T).$$
Let $\lambda(\alpha)$ be the principal eigenvalue of  \eqref{LTD_eq 1}. Then
$\lambda(\alpha)\to\lambda^*$ as $\alpha\to\infty$,
where $\lambda^*$ is the principal eigenvalue of the  problem
\begin{equation}\label{LTD_auxi 1}
\begin{cases}
\partial_{t}\psi-\partial_{xx}\psi+V\psi=\lambda\psi &{\mathrm{in}}\,\,(0,1)\times(0,t_*],\\
\partial_{t}\psi+V(1,t)\psi=\lambda \psi &{\mathrm{on}}\,\,(t_*,T],\\
\psi(x,t_*^+)\equiv\psi(1,t_*^{-}) &{\mathrm{on}}\,\,(0,1),\\
\partial_{x}\psi(0,t)=\partial_{x}\psi(1,t)=0& {\mathrm{on}}\,\,[0,T],\\
\psi(x,0)=\psi(x,T) &{\mathrm{on}}\,\,(0,1).
\end{cases}
\end{equation}
\end{prop}

\begin{remark}\label{rek4.1}
{\rm The existence and uniqueness of the principal eigenvalue $\lambda^*$ for problem \eqref{LTD_auxi 1} is proved in Proposition {\rm\ref{principaleigen}}.
Denote by $\psi^*>0$ the corresponding principal  eigenfunction. We mention that $\psi^*(t,\cdot)$ is constant
for any fixed $t\in\left(t_*,T\right]$, and  $\psi^*$ is not necessarily continuous
at  $t=t_*$ 
with respect to $t$, while it is differentiable elsewhere.
}
\end{remark}

\begin{proof}[Proof of Proposition {\rm \ref{LTD_thm 1}}]
We first prove  $\liminf\limits_{\alpha\rightarrow\infty}\lambda(\alpha)\geq \lambda^*$.
Given any $\epsilon>0$, we shall construct a  function $\overline{\varphi}>0$  such that for sufficiently large $\alpha$,
\begin{equation}\label{LTD1}
\begin{cases}
\mathcal{L}_{\alpha}\overline{\varphi}
\geq(\lambda^*-\epsilon)\overline{\varphi}\ \ &{\text{in}}\,\,(0,1)\times\left((0,T)\setminus\{t_*,t_*+\delta\}\right),\\
 \partial_{x}\overline{\varphi}(0,t)\leq0, \ \ \partial_{x}\overline{\varphi}(1,t)\geq0  &{\text{in}}\,\,[0,T], \\
 \overline{\varphi}(x,t_*^+)>\overline{\varphi}(x,t_*^-)  & {\text{on}}\,\,(0,1), \\
  \partial_t\overline{\varphi}(x,  (t_*+\delta)^-)>\partial_t\overline{\varphi}(x,(t_*+\delta)^+)  & {\text{on}}\,\,(0,1), \\
 \overline{\varphi}(x,0)=\overline{\varphi}(x,T) &{\text{in}}\,\,[0,1], \\
 \end{cases}
\end{equation}
where the constant $\delta>0$ will be determined later. Then such function  $\overline{\varphi}$ is
 a strict super-solution in the sense of Definition \ref{appendixldef} 
 with
 $\mathbb{X}=\emptyset$ and $\mathbb{T}=(0,1)\times\{t_*,  t_*+\delta\}$.
 We may apply  Proposition \ref{appendixprop} to conclude $\liminf\limits_{\alpha\rightarrow\infty}\lambda(\alpha)\geq \lambda^*$.
For this purpose, we define
\begin{equation}\label{liu005}
    \overline{\varphi}(x,t):=\gamma(t)\beta(x,t)\psi^*(x,t).
\end{equation}
Here $\beta(x,t)>0$ is a $T$-periodic function such that $\beta\in C^1([0,1]\times([0,t_*)\cup(t_*,T])$ and
$$\begin{cases}
\beta=1& \mathrm{in}~ [0,1]\times[0,t_*),\\
\beta=1, \,\, \partial_x\beta=0 & \mathrm{on}~ \{1\}\times(t_*,T],\\
\partial_x\beta<0 &\mathrm{in}~ [0,1)\times(t_*,T).
\end{cases}$$
Note that the chosen $\beta$ may not be continuous at $t=t_*$.

Since $\partial_x\psi^*(1,t)\equiv0$ holds in the neighborhood of $t=t_*$,
we can require $\partial_{xx}\beta(1,t_*^+)$ to be sufficiently large such that
$\overline{\varphi}(x,t_*^+)>\overline{\varphi}(x,t_*^-)$ holds in the neighborhood of $x=1$.
Thus, we can always find the appropriate function $\beta$ and constant $M_\beta>0$ such that
$$\begin{cases}
\smallskip
\overline{\varphi}(x,t_*^+)>\overline{\varphi}(x,t_*^-) &\text{ on } [0,1),\\
|\partial_t\beta|,\ |\partial_x\beta|,\ |\partial_{xx}\beta|<M_\beta\beta &\text{ in } [0,1]\times[0,T].
\end{cases}$$

Then we  select the function $\gamma:\ [0,T]\mapsto(0,\infty)$  to be continuous and satisfy
\begin{equation}\label{LTD2}
\gamma(0)=\gamma(T)=1,\quad \gamma'\geq0\ \,\mathrm{on}\,\left[0,t_*+\delta\right) \cup (T-\delta,T), \quad \text{and} \quad
\gamma'<0\,\, \mathrm{ on}~\left(t_*+\delta,T-\delta\right),
\end{equation}
and the following:
\begin{equation}\label{LTD2.1}
    \begin{cases}
    \gamma'(0^+)=0, \,\,\gamma'(T^-)=M_\gamma,\\
    \gamma'(t_*^-)>\gamma'(t_*^+),\\
\gamma'((t_*+\delta)^-)>0>\gamma'((t_*+\delta)^+),\\
\gamma'(T-\delta)=0.
\end{cases}
\end{equation}
Here the positive constants $\delta$ and $M_\gamma$ are determined as follows:

For small $0<\delta<(T-t_*)/2$, we set
\begin{equation}\label{definM}
   M_\gamma:=\frac{2\ln(1+\epsilon)}{3\delta} \,\,\text{ and }\,\frac{1}{Q}:=\frac{8\ln{(1+\epsilon)}}{\epsilon(T-t_*)}.
\end{equation}
Letting $\delta$ be small enough if necessary so that the followings hold:
\begin{itemize}
  \item[{\rm{(i)}}] $M_\gamma\geq2M_\beta+\max\limits_{[0,1]\times[0,T]}\left[V(1,t)-V(x,t)\right];$
    \smallskip
  \item[{\rm{(ii)}}] $|V(x,t)-V(1,t)|<\frac{\epsilon}{3}$ in $[1-2\delta,1]\times[0,T];$
  \medskip
  \item[{\rm{(iii)}}] 
$3\frac{M_\gamma\delta}{\epsilon}+\frac{2\delta}{Q}<\frac{3(T-t_*)}{4Q}.$
\end{itemize}
We will specify the function $\gamma$ later, whose profile can be exhibited in Fig. \ref{figureb1}.
\begin{figure}[http!!]
  \centering
\includegraphics[height=1.5in]{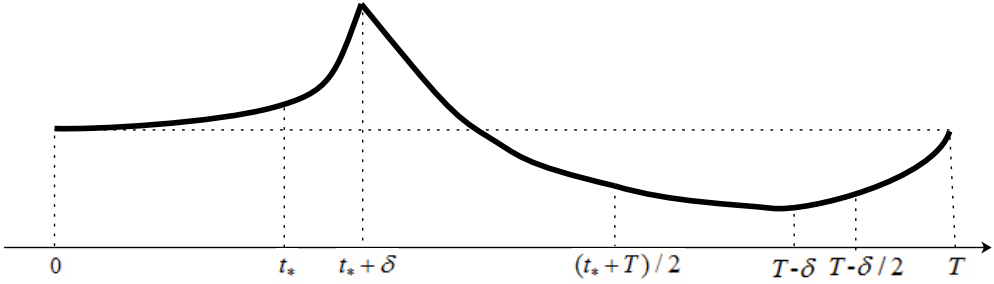}
  \caption{\small  Profile of the function $\gamma$, which is illustrated by the black solid curve.}\label{figureb1}
  \end{figure}

To verify that $\overline{\varphi}$ defined above is a super-solution satisfying \eqref{LTD1}, it remains to verify $\mathcal{L}_\alpha\overline{\varphi}\geq(\lambda^*-\epsilon)\overline{\varphi}$
in $(0,1)\times\left((0,T)\setminus\{t_*,t_*+\delta\}\right)$. To this end,
we divide the time interval $[0,T]$ into four parts:
$\left[0,t_*\right]$, $\left[t_*,t_*+\delta\right]$, $\left[t_*+\delta,T-\delta\right]$, and $[T-\delta,T]$.

\smallskip

\noindent{\bf Part 1.} For $t\in \left[0,t_*\right]$, we choose $\gamma$ to satisfy
\begin{equation}\label{gamma_1}
    \gamma'(t)\geq0, \quad \gamma(0)=1,\quad \gamma'(0^+)=0,\quad \gamma(t_*)=1+\epsilon.
\end{equation}
In view of $b(t)=0$, $\beta(x,t)=1$,  and
$\gamma'(t)\geq0\,\text{ in } \,[0,1]\times\left[0,t_*\right]$, we use \eqref{liu005} to calculate that
\begin{align*}
\mathcal{L}_\alpha\overline{\varphi}&=\gamma\partial_t\psi^*
+\gamma'\psi^*-\gamma\partial_{xx}\psi^*+\gamma V(x,t)\psi^*
\\&\geq\gamma(\partial_t\psi^*-\partial_{xx}\psi^*+V(x,t)\psi^*)\\
&=\lambda^*\gamma\psi^*>(\lambda^*-\epsilon)\overline{\varphi},\ \ \forall (x,t)\in[0,1]\times[0,t_*],
\end{align*}
where the second equality is due to  the definition of $\psi^*$ in \eqref{LTD_auxi 1}.

\smallskip

\noindent{\bf Part 2.} For $t\in\left(t_*,t_*+\delta\right]$,  we define
 \begin{equation}\label{gamma_2}
\begin{array}{l}
\gamma(t):=(1+\epsilon)e^{M_\gamma\left(t-t_*\right)},
\end{array}
\end{equation}
where $M_\gamma$ is defined in \eqref{definM}.
 Recall that for any fixed
 $t\in\left[t_*,T\right]$, $\psi^*(\cdot,t)$ is a constant  as noted in Remark \ref{rek4.1}.  By direct calculation we have
$$
\mathcal{L}_\alpha\overline{\varphi}= \gamma\beta\partial_t\psi^*+V(1,t)\gamma\beta\psi^*+\gamma'\beta\psi^*+[\partial_t\beta-\partial_{xx}\beta-\alpha b\partial_x\beta+V(x,t)\beta-V(1,t)\beta]\gamma\psi^*.
$$
Since $\partial_x\beta\leq 0$ and $|\partial_t\beta|,|\partial_x\beta|,|\partial_{xx}\beta|<M_\beta\beta$
on $[0,1]\times(t_*,t_*+\delta]$,  we deduce from (i) that
\begin{align*}
\mathcal{L}_{\alpha}\overline{\varphi}&\geq \lambda^*\gamma\beta\psi^*+\gamma'\beta\psi^*+[\partial_t\beta-\partial_{xx}\beta+V(x,t)\beta-V(1,t)\beta]\gamma\psi^*\\
&\geq \lambda^*\overline{\varphi}+\gamma'\beta\psi^*+[-2M_\beta+V(x,t)-V(1,t)]\gamma\beta\psi^*\\
&=[\lambda^*-2M_\beta+V(x,t)-V(1,t)+M_\gamma]\overline{\varphi}\\
&\geq (\lambda^*-\epsilon)\overline{\varphi},\ \ \forall (x,t)\in[0,1]\times(t_*,t_*+\delta].
\end{align*}

\smallskip

\noindent{\bf Part 3.} For $t\in\left[t_*+\delta,T-\delta\right]$, we set
\begin{equation}\label{LTD_3}
    \begin{cases}
\gamma(t)=(1+\epsilon)e^{M_\gamma\delta}e^{-\frac{\epsilon}{Q}(t-t_*-\delta)}&\text{ on }\,\,\left[t_*+\delta,\frac{T+t_*}{2}\right],\\
-\frac{2\epsilon\gamma}{Q}<\gamma'(t)<0, &\text{ on }\,\,\left(\frac{T+t_*}{2},T-\delta\right),\\
\gamma'(T-\delta)=0,
\end{cases}
\end{equation}
where $M_\gamma$ and $Q$ are defined in \eqref{definM}.
We then calculate that
\begin{equation}\label{LTD_3.1}
\begin{split}
\mathcal{L}_{\alpha}\overline{\varphi}=\lambda^*\gamma\beta\psi^*+\gamma'\beta\psi^*+[\partial_t\beta-\partial_{xx}\beta-\alpha b\partial_x\beta+V(x,t)\beta-V(1,t)\beta]\gamma\psi^*.
\end{split}
\end{equation}
Note that $\beta=1$ on  $\{1\}\times\left[t_*+\delta,T\right]$. By taking $\delta$ small if necessary, we may assume
\begin{equation}\label{condition_1}
 |\partial_t\beta|+|\partial_{xx}\beta|<\tfrac{\epsilon}{3}\beta \quad \text{on}\
 [1-2\delta,1]\times\left[t_*+\delta,T\right].
\end{equation}

For $(x,t)\in[1-2\delta,1]\times\left[t_*+\delta,T-\delta\right]$, in view of $\partial_x\beta\leq0$ and (ii), it follows from  \eqref{LTD_3.1} that
\begin{align*}
\mathcal{L}_{\alpha}\overline{\varphi}\geq \lambda^*\overline{\varphi}+\gamma'\beta\psi^*-\epsilon\gamma\beta\psi^*.
\end{align*}
By the choice of $\gamma$ in \eqref{LTD_3}, we arrive at
\begin{align*}
\mathcal{L}_{\alpha}\overline{\varphi}\geq\left(\lambda^*-\tfrac{(Q+2)\epsilon}{Q}\right)\overline{\varphi}.
\end{align*}

 For $(x,t)\in\left[0,1-2\delta\right]\times\left[t_*+\delta,T-\delta\right]$,
 by the conditions on $\beta$ and $b$, clearly there exists some constant $\rho>0$ independent of  $(x,t)$ such that $b(t)\partial_x\beta(x,t)<-\rho$, whence by \eqref{LTD_3.1},
\begin{equation*}
\begin{array}{l}
\mathcal{L}_{\alpha}\overline{\varphi}\geq \left[\lambda^*\beta-\frac{2}{Q}\epsilon\beta-2M_\beta \beta+V(x,t)\beta-V(1,t)\beta
+\alpha\rho\right]\gamma\psi^*. 
\end{array}
\end{equation*}
By choosing $\alpha$  large, we deduce $\mathcal{L}_{\alpha}\overline{\varphi}\geq (\lambda^*-\epsilon)\overline{\varphi}$ as desired.

\smallskip
\noindent{\bf Part 4.} For $t\in \left[T-\delta,T\right]$, it can be observed from  \eqref{definM}  and {\rm (iii)} that
\begin{equation*}
\gamma\left(\tfrac{T+t_*}{2}\right)=(1+\epsilon)e^{M_\gamma\delta}e^{-\frac{\epsilon}{Q}
\left(\frac{T-t_*}{2}-\delta\right)}<e^{-\frac{M_\gamma\delta}{2}}=\gamma\left(T-\delta\right).
\end{equation*}
This allows us to define $\gamma$ as a smooth function on $(T-\delta,T]$ such that
\begin{equation}\label{LTD_4}
\gamma'\geq0\ \ \text{ on } (T-\delta,T-\tfrac{\delta}{2})\quad \text{and}\quad \gamma=e^{M_\gamma(t-T)} \,\, \text{ on } [T-\tfrac{\delta}{2},T].
\end{equation}
To verify $\mathcal{L}_{\alpha}\overline{\varphi}\geq(\lambda^*-\epsilon)\overline{\varphi},$
we consider the following different regions.

For $(x,t)\in\left[1-2\delta,1\right]\times\left[T-\delta,T\right]$, by means of the facts $|\partial_t\beta|,|\partial_{xx}\beta|<\frac{\epsilon}{4}\beta$ in \eqref{condition_1}, $\partial_x\beta\leq0$ and $\gamma'(t)\geq0$,
we deduce from \eqref{LTD_3.1} and {\rm(ii)} that
\begin{equation*}
\begin{array}{l}
\medskip
\mathcal{L}_{\alpha}\overline{\varphi}
\geq \lambda^*\gamma\beta\psi^*-\epsilon \gamma\beta\psi^*= (\lambda^*-\epsilon)\overline{\varphi}.
\end{array}
\end{equation*}

For $(x,t)\in\left[0,1-2\delta\right]\times\left[T-\frac{\delta}{2},T\right]$,
in view of $\partial_x\beta<0$ and $\gamma=e^{M_\gamma(t-T)}$, it follows from \eqref{LTD_3.1} and {\rm(i)} that
\begin{align*}
\mathcal{L}_\alpha\overline{\varphi}&\geq\lambda^*\gamma\beta\psi^*+\gamma'\beta\psi^*
+\left[\partial_t\beta-\partial_{xx}\beta+V(x,t)\beta-V(1,t)\beta\right]\gamma\psi^*\\
&\geq \left[\lambda^*-2M_\beta+V(x,t)-V(1,t)+M_\gamma\right]\overline{\varphi}\geq (\lambda^*-\epsilon)\overline{\varphi}.
\end{align*}

For $(x,t)\in\left[0,1-2\delta\right]\times\left[T-\delta,T-\frac{\delta}{2}\right]$,
there exists some constant $\rho>0$ independent of $(x,t)$ such that $b(t)\partial_x\beta(x,t)<-\rho$ and $\gamma'>0$.
Noting that $|\partial_t\beta|,|\partial_{xx}\beta|<M_\beta\beta$, by choosing $\alpha$ to be large enough,  we may use \eqref{LTD_3.1} 
to  obtain
\begin{align*}
\mathcal{L}_\alpha\overline{\varphi}
&\geq \left[\lambda^*\beta-2M_\beta\beta+V(x,t)\beta-V(1,t)\beta+\alpha\rho\right]\gamma\psi^*\geq (\lambda^*-\epsilon)\overline{\varphi}.
\end{align*}

 By now, we have specified the function $\gamma$ through \eqref{gamma_1}, \eqref{gamma_2},
 \eqref{LTD_3} and \eqref{LTD_4}, which satisfies \eqref{LTD2} and \eqref{LTD2.1}.
 Therefore, the super-solution $\overline{\varphi}$ constructed above satisfies \eqref{LTD1}, and thus $\liminf\limits_{\alpha\rightarrow\infty}\lambda(\alpha)\geq \lambda^*$ is established.

The proof of  $\limsup\limits_{\alpha\rightarrow\infty}\lambda\leq \lambda^*$ is similar, which amounts to construct a sub-solution $\underline{\varphi}$ such that
\begin{equation}\label{LTDsub}
\begin{cases}
 \mathcal{L}_{\alpha}\underline{\varphi}\leq(\lambda^*+\epsilon)\underline{\varphi} &{\text{in}}\,\,(0,1)\times((0,T)\setminus\{t_*,t_*+\delta\}),\\
 \partial_{x}\underline{\varphi}(0,t)\geq0, \ \ \partial_{x}\underline{\varphi}(1,t)\leq0  &{\text{on}}\,\,[0,T],\\
   \underline{\varphi}(x,t_*^+)<\underline{\varphi}(x,t_*^-)  & {\text{on}}\,\,(0,1), \\
    \partial_t\overline{\varphi}(x,(t_*+\delta)^-)<\partial_t\overline{\varphi}(x,(t_*+\delta)^+)  & {\text{on}}\,\,(0,1), \\
 \underline{\varphi}(x,0)=\underline{\varphi}(x,T) &{\text{on}}\,\,(0,1),
 \end{cases}
\end{equation}
where $\delta$ is given as before. Indeed, we can define
\begin{equation*}
\underline{\varphi}(x,t):=\frac{\underline{\beta}(x,t)\psi^*}{\gamma(t)},
\end{equation*}
where $\gamma$ is defined as above and $\underline{\beta}>0$ has the same properties
as $\beta$ except that $\partial_x \underline\beta>0$ on $ [0,1)\times(t_*,T)$.
 Proceeding as before, we can verify such a function $\underline{\varphi}$
 is a strict sub-solution satisfying \eqref{LTDsub}.
Applying Proposition \ref{appendixprop} again,  we derive
$\limsup_{\alpha\rightarrow\infty}\lambda(\alpha)\leq \lambda^*$. The proof is now completed.
\end{proof}

Next, we consider another typical function $b$.
\begin{prop}\label{LTD_thm 2}  Assume that the $T$-periodic function $b\in C(\mathbb{R})$ satisfies
$$
b(t)>0 \text{ in }(0,t_*), \quad b(t)<0 \text{ in }(t_*,T), \quad
\text{and} \quad b(0)=b\left(t_*\right)=b(T)=0 \quad {\text for\  some}\ t_*\in(0,T).$$
Let $\lambda(\alpha)$ be the principal eigenvalue of  \eqref{LTD_eq 1}. Then
$$\lim_{\alpha\rightarrow\infty}\lambda(\alpha)=\frac{1}{T}
\left[\int^{t_*}_0V(1,s)\mathrm{d}s+\int^T_{t_*}V(0,s)\mathrm{d}s\right].$$
\end{prop}

\begin{proof}  Consider the following auxiliary eigenvalue problem:
\begin{equation}\label{LTD_auxi eigenfunction}
\left\{\begin{array}{ll}
\frac{{\rm d}\psi}{{\rm d} t}+V(1,t)\psi=\lambda \psi &{\text{in}}\,\,\left(0,t_{*}\right],  \\
\frac{{\rm d}\psi}{{\rm d} t}+V(0,t)\psi=\lambda \psi &{\text{in}}\,\,\left(t_*,T\right],\\
 \psi(t_{*}^+)=\psi(t_{*}^-),\\
\psi(0)=\psi(T).
\end{array}\right.
\end{equation}
The existence of the principal eigenvalue, denoted by $\lambda^{**}$, of problem \eqref{LTD_auxi eigenfunction}
is shown in Proposition \ref{principaleigen}. It is easily seen that
 $
 \lambda^{**}=\frac{1}{T}\left[\int^{t_*}_0V(1,s)\mathrm{d}s+\int^T_{t_*}V(0,s)\mathrm{d}s\right].
 $
Define $\psi^{**}(t)>0$ as the corresponding principal eigenfunction of \eqref{LTD_auxi eigenfunction}.
Set
$$\overline{\varphi}(x,t):=\overline{\gamma}(t)\overline{\beta}(x,t)\psi^{**}(t).$$
Here $\overline{\beta}\in C^{1,1}([0,1]\times([0,t_*)\cup(t_*,T]))
$  is chosen to be   $T$-periodic and satisfies
$$\begin{cases}
\overline{\beta}=1,\, \, \partial_x\overline{\beta}=0\ \ &\text{on}~(\{1\}\times(0,t_*))\cup(\{0\}\times(t_*,T)),\\
\partial_x\overline{\beta}<0 &\text{in}~[0,1)\times(0,t_*),\\
\partial_x\overline{\beta}>0 &\text{in}~(0,1]\times(t_*,T),\\
\overline{\beta}=1& \text{on}~[0,1]\times\{t_*\}. 
\end{cases}$$
 Given $\delta>0$ to be determined later, the $T$-periodic function $\overline{\gamma}\in C([0,T])$ is chosen to be positive, smooth except for $t\in\{t_*\pm\delta, t_*\}$, and satisfy the following properties:
\begin{equation*}%
\begin{cases}
\overline{\gamma}'>0\ \ &\text{in}~ \left[0, \delta\right)\cup \left(t_*-\delta, t_*+\delta\right)\cup\left(T-\delta, T\right],\\
\overline{\gamma}'<0&\text{in}~ \left(\delta, t_*-\delta\right)\cup\left(t_*+\delta, T-\delta\right),\\
\overline{\gamma}'=0&\text{on}~\{\delta, T-\delta\},
\end{cases}
\end{equation*}
and
$\overline{\gamma}'(T^-)>\overline{\gamma}'(0^+)$ and
$\overline{\gamma}'(t^-)>\overline{\gamma}'(t^+) $ for $t=t_*\pm \delta, \, t_*$.


Given any $\epsilon>0$, we can find suitable $\delta$, $\overline{\gamma}$ and $\overline{\beta}$ such that the above $\overline{\varphi}$ satisfies
\begin{equation*}
 \begin{cases}
\mathcal{L}_\alpha\overline{\varphi}\geq(\lambda^{**}-\epsilon)\overline{\varphi}\ \ &{\text{in}}\,\,((0,1)\times[0,T])\setminus\mathbb{T},\\
 \overline{\varphi}(x,t_*^+)>\overline{\varphi}(x,t_*^-)  & {\text{on}}\,\,(0,1), \\
  \partial_t\overline{\varphi}(x,t^-)>\partial_t\overline{\varphi}(x,t^+)
 &{\text{on}}\,\, (0,1)\times\{t_*-\delta,\,t_*+\delta\},\\
  \partial_{x}\overline{\varphi}(0,t)\leq0, \ \ \partial_{x}\overline{\varphi}(1,t)\geq0 &{\text{on}}\,\,[0,T], \\
 \overline{\varphi}(x,0)=\overline{\varphi}(x,T) &{\text{on}}\,\,(0,1),
 \end{cases}
\end{equation*}
with $\mathbb{T}=[0,1]\times\{t_*\pm\delta,\,t_*\}$, provided that $\alpha$ is large enough. The process is a straightforward adaptation of the proof presented in  parts $\mathbf{(2)}$-$\mathbf{(4)}$ of Proposition \ref{LTD_thm 1}, and we  omit the details. Consequently, applying Proposition \ref{appendixprop} and letting $\epsilon\searrow 0$ give
$\liminf_{\alpha\rightarrow\infty}\lambda(\alpha)\geq\lambda^{**}.$

The construction of a sub-solution $\underline{\varphi}$ follows from the same arguments as in
 Proposition \ref{LTD_thm 1} and thus
$\limsup_{\alpha\rightarrow\infty}\lambda(\alpha)\leq\lambda^{**}.$  Proposition \ref{LTD_thm 2} is now proved.
\end{proof}

By a similar analysis to  Proposition \ref{LTD_thm 2}, we can also deduce the following result.

\begin{prop}\label{LTD_thm 3} Assume that the $T$-periodic function $b\in C(\mathbb{R})$ satisfies
\begin{equation*}
b(t)>0\ \  \text{ in }\,(0,t_*)\cup(t_*,T) \quad \text{and} \quad b\left(t_*\right)=0 \quad {\text for\  some}\ t_*\in(0,T).
\end{equation*}
Let $\lambda(\alpha)$ be the principal eigenvalue of  \eqref{LTD_eq 1}. Then
$\lambda(\alpha)\to\frac{1}{T}\int^{T}_0V(1,s)\mathrm{d}s$ as $\alpha\to \infty$.
\end{prop}

\begin{remark}\label{re4.a} 
{\rm {\rm(1)} The result parallel to Proposition {\rm\ref{LTD_thm 2}} holds: Let $b$ satisfy
    $$
    b(t)<0 \text{ in }(0,t_*), \,\, b(t)>0 \text{ in }(t_*,T),
    \text{ and } b(0)=b\left(t_*\right)=b(T)=0 \quad \text {for  some} \ t_*\in(0,T).$$
    Then the principal eigenvalue $\lambda(\alpha)$ of \eqref{LTD_eq 1} satisfies
    $$\lim_{\alpha\rightarrow\infty}\lambda(\alpha)=\frac{1}{T}
    \left[\int^{t_*}_0V(0,s)\mathrm{d}s+\int^T_{t_*}V(1,s)\mathrm{d}s\right].$$
   This can be deduced by using the variable change $y=1-x$ in Proposition {\rm\ref{LTD_thm 2}}.    Similarly, we can derive the results parallel to Propositions {\rm\ref{LTD_thm 1}} and {\rm\ref{LTD_thm 3}};

{\rm (2)} In Proposition {\rm\ref{LTD_thm 1}}, the limit value $\lambda^*$ (of the principal eigenvalue as $\alpha\to\infty$) satisfies
  $\lambda^*\to\lambda^{\mathcal{N}\mathcal{N}}\big((0,1)\big)$ as $t_*\nearrow T$ and
  $\lambda^*\to\frac{1}{T}\int^{T}_0V(1,s)\mathrm{d}s$ as $t_*\searrow0$, which correspond to Proposition {\rm \ref{ldnlem2}} and Remark {\rm \ref{rem2.1}}, respectively; 

  {\rm (3)}  Proposition {\rm\ref{LTD_thm 3}} suggests that  when $b\geq\not\equiv0$ or $b\leq\not\equiv0$ on $[0,T]$,
  finitely many isolated critical points of $b$
  will have no effects on the limit of $\lambda(\alpha)$ as $\alpha\rightarrow\infty$.
  }
 \end{remark}

By the ideas developed in the proofs of Propositions \ref{LTD_thm 1} and \ref{LTD_thm 2},
we  can derive Theorem \ref{LTD_thm main}.

\begin{proof}[Proof of Theorem {\rm \ref{LTD_thm main}}]
As before, it suffices to construct the suitable sub-solution and super-solution for \eqref{LTD_eq 1} and apply Proposition \ref{appendixprop}.
We 
shall follow the above arguments 
to construct a super-solution here and a sub-solution can be found similarly.

Let $\psi_\infty>0$ denote the principal eigenfunction of problem \eqref{LTD_auxi main}. Define
\begin{equation}\label{liu019}
 \overline{\varphi}(x,t):=\gamma(t)\beta(x,t)\psi_\infty(x,t).
\end{equation}
 Following the proofs of Propositions \ref{LTD_thm 1} and \ref{LTD_thm 2}, the $T$-periodic function $\beta$ is chosen to satisfy
 $\beta\in C^{1,1}([0,1]\times([0,T]\setminus \{t_i, 1\leq i\leq N\}))$ and
$$\begin{cases}
\beta=1,\partial_x\beta=0, &\mathrm{on}~(\{1\}\times\left\{\left[t_i,t_{i+1}\right],i\in \mathbb{C}\right\})\cup(\{0\}\times\left\{\left[t_i,t_{i+1}\right],i\in \mathbb{A}\right\}),\\
\partial_x\beta<0 &\text{on}~\left[0,1\right)\times\left\{\left[t_i,t_{i+1}\right],i\in \mathbb{C}\right\},\\
\partial_x\beta>0 &\text{on}~\left(0,1\right]\times\left\{\left[t_i,t_{i+1}\right],i\in \mathbb{A}\right\},\\
\beta=1& \text{on}~\left[0,1\right]\times\left(\left\{\left[t_i,t_{i+1}\right],i\in \mathbb{B}\right\}\cup\left\{t_i,i\in \mathbb{A}\cup \mathbb{C}\right\}\right).
\end{cases}$$
The $T$-periodic function $\gamma$ is assumed to be positive, continuous everywhere and satisfy that 
\begin{equation*}
    \begin{cases}
\gamma'(t)\geq0\  & \text{on}\,\left\{\left[t_i,t_{i+1}\right],i\in \mathbb{B}\right\},\\
\gamma'(t)>0 &  \text{on}\,\left\{\left(t_i,t_i+\delta\right) \cup (t_{i+1}-\delta,t_{i+1}),i\in \mathbb{A}\cup \mathbb{C}\right\},\\
\gamma'(t)<0 & \text{on}\,\left\{\left(t_i+\delta,t_{i+1}\right),i\in \mathbb{A}\cup \mathbb{C}\right\},
\end{cases}
\end{equation*}
and that for any $1\leq i\leq N+1$,
\begin{equation*}
\gamma'(t_{i+1}-\delta)=0, \quad \gamma'(t_i^-)>\gamma'(t_i^+),\quad \text{and}\quad
\gamma'((t_i+\delta)^-)>\gamma'((t_i+\delta)^+).
\end{equation*}

Let $\mathbb{T}=(0,1)\times\{t_i, t_i\pm\delta, 1\leq i\leq N \}$. By the similar arguments as in Proposition \ref{LTD_thm 1}, we can piecewise construct suitable  $\gamma$ and $\beta$ on different intervals $[t_i,t_{i+1}]$ ($0\leq i\leq N$) 
such that for small $\epsilon>0$,
 the chosen $\overline{\varphi}$ in \eqref{liu019} satisfies
\begin{equation*}
\begin{cases}
\mathcal{L}_\alpha\overline{\varphi}\geq(\lambda_{\infty}-\epsilon)\overline{\varphi} &{\rm{in}}\,\left((0,1)\times[0,T]\right)\setminus\mathbb{T},\\
 \overline{\varphi}(x,t_i^+)>\overline{\varphi}(x,t_i^-)  & {\text{on}}\,\,(0,1),\, 1\leq i\leq N, \\
  \partial_t\overline{\varphi}(x,t^-)>\partial_t\overline{\varphi}(x,t^+)
 &{\text{on}}\,\, (0,1)\times\{t_i\pm\delta, 1\leq i\leq N \},\\
 \partial_{x}\overline{\varphi}(0,t)\leq0, \ \ \partial_{x}\overline{\varphi}(1,t)\geq0 &{\text{in}}\,[0,T], \\
 \overline{\varphi}(x,0)=\overline{\varphi}(x,T) &{\text{in}}\,(0,1),
 \end{cases}
\end{equation*}
provided that $\alpha$ is sufficiently large. This implies that $\overline{\varphi}$  is a strict super-solution in the sense of Definition \ref{appendixldef}. Therefore, Theorem \ref{LTD_thm main} follows from  Proposition \ref{appendixprop}.
\end{proof}

\appendix

\section{The existence of the principal eigenvalue of \eqref{LTD_auxi main}}\label{appenB}
This section is devoted to the proof of  the existence and uniqueness of the principal eigenvalue for problem \eqref{LTD_auxi main}, which is the limit value for $\lambda(\alpha)$ as $\alpha\to\infty$ under the assumption there.

\begin{prop}\label{principaleigen}
The problem \eqref{LTD_auxi main} admits a unique principal eigenvalue $\lambda_\infty$ that corresponds to an eigenfunction $\psi$ satisfying
\begin{equation}\label{principaleigenfunction}
\tag{A.1}
\left\{
\begin{array}{l}
\medskip
  \psi>0 \,\, \text{ in }\, [0,1]\times[0,T], \\
    \psi\in  C^{2+\sigma,1+\frac{\sigma}{2}}
    \left([0,1]\times([0,T]\setminus\{t_1,\ldots,t_{N+1}\})\right)\,\text{ for some }\sigma\in(0,1).
\end{array}\right.
\end{equation}
 Conversely, if \eqref{LTD_auxi main} has a solution $\psi_1$ satisfying
 problem \eqref{principaleigenfunction}, then necessarily
$\lambda=\lambda_\infty$  and $\psi_1=c\psi$ for some constant $c>0$.
\end{prop}

\begin{proof}
Our proof essentially adapts the ideas developed in \cite[Theorem 3.4]{DP2012}, which deals with an eigenvalue problem over a varying cylinder.

For any given $u_0(x,0)\in X:=\{w\in C^1([0,1]): w'(0)=w'(1)=0\}$,
we define $u_i(x,t)$ ($1\leq i\leq N$) recursively as the unique solution of the problem
\begin{equation}\label{u_i}
\tag{A.2}
\begin{cases}
\left.\begin{array}{ll}
\smallskip
 \partial_{t}u_i+V(0,t)u_i=0\ \ \ \ \ \ &{\text{in}}\,\,(0,1)\times\left(t_i,t_{i+1}\right],  \\
             u_i(x,t_{i}^+)\equiv u_{i-1}(0,t_{i}^-)&{\text{on}}\,\,(0,1),\\
              \end{array} \right\}\ \,i\in \mathbb{A} \\
\left.\begin{array}{ll}
\smallskip
 \partial_{t}u_i-\partial_{xx}u_i+V(x,t)u_i=0 &{\text{in}}\,\,(0,1)\times\left(t_i,t_{i+1}\right],  \\
              u_i(x,t_{i}^+)=u_{i-1}(x,t_{i}^-) &{\text{on}}\,\,(0,1),\\
             \end{array}\right\}\ \,i\in \mathbb{B} \\
\left.\begin{array}{ll}
\smallskip
\partial_{t}u_i+V(1,t)u_i=0\ \ \ \ \ \ &{\text{in}}\,\,(0,1)\times\left(t_i,t_{i+1}\right],  \\
             u_i(x,t_{i}^+)\equiv u_{i-1}(1,t_{i}^-)&{\text{on}}\,\,(0,1),\\
              \end{array}\right\}\ \,i\in \mathbb{C} \\
\,\,\,\partial_{x}u_i(0,t)=\partial_{x}u_i(1,t)=0\ \ {\text{in}}\,\left[0,T\right],
\end{cases}
\end{equation}
where  $0=t_1<t_2<\ldots<t_{N+1}=T$ and the sets $\mathbb{A},\mathbb{B},\mathbb{C}$ are given in Theorem \ref{LTD_thm main}. By the standard $L^p$-theory of parabolic equations \cite{F1964}, we know that
$$u_i\in C^\sigma\left((t_i,t_{i+1}], W^{2,p}(0,1)\right)\cap C^{1+\sigma}\left((t_i,t_{i+1}], L^{p}(0,1)\right)\,\, \text{ for some }\sigma>0.$$
We may choose $p$ large enough such that $W^{2,p}(0,1)$ is embedded into $C^1([0,1])$.
Thus, $u_i(\cdot,t)\in X$ for any $t\in[0,T]$ and $1\leq i\leq N$.

Let $u_0$ and $u_i$ ($1\leq i\leq N$) be given above. Define the operator $K:X\rightarrow X$ by
$$Ku_0:=u_N(\cdot,T).$$
We use $P^o$ to denote the interior of $P$, a cone of nonnegative functions in $X$.
Then it can be shown that the operator $K$ is linear, compact and strongly positive.
This fact can be verified by a standard argument with the help of the regularity theory
and the maximum principle for parabolic equations. We omit the details and refer to
  \cite[Theorem 3.4]{DP2012}.

By the above properties for $K$, it follows from
the  Krein-Rutman theorem \cite{KR1950}
that the spectral radius $r(K)$ of $K$ is positive, and it corresponds to
an eigenvector $u^{*}\in P^{o}$. Moreover, if $K \tilde u =\tilde r\tilde u$
for some $\tilde u\in P^o$, then necessarily $\tilde r=r(K)$ and $\tilde u=cu^{*}$ for some constant $c>0$.

Let $U:[0,1]\times[0,T]$ be given by
$$U(x,0)=u^{*}(x) \text{ in } [0,1] \quad \text{and}\quad U(x,t)=u_i(x,t)
\text{  in  } [0,1]\times(t_i,t_{i+1}] \quad \text{ for }1\leq i\leq N,$$
where $u_i$ is defined by \eqref{u_i} with $u_0(x)=u^{*}(x)$. By definition, we have
$$U(\cdot,T)=Ku^*=r(K)u^*\,\text{ in }[0,1].$$
Set $\psi^*(x,t):=e^{\lambda_\infty t} U(x,t)$ with $\lambda_\infty=-\frac{1}{T}\ln r(K)$.
Clearly, such a function $\psi^*$ satisfies \eqref{principaleigenfunction} and $\psi^*(x,0)=\psi^*(x,T)$.  Direct calculations give that $\psi^*$ verifies \eqref{LTD_auxi main} with $\lambda=\lambda_\infty$, which proves the existence of the principal eigenvalue.

 Conversely, if \eqref{LTD_auxi main} has a solution $\psi_1$ satisfying \eqref{principaleigenfunction}, then let
 $\tilde r:=e^{-\lambda T}$ and $\tilde u(x,t):=e^{-\lambda t}\psi_1(x,t).$
 It can be verified that $\tilde u(x,t)$ satisfies \eqref{u_i} with $u_0=\psi_1(\cdot,0)$, and furthermore
 $$K\psi_1(\cdot,0)=\tilde u(\cdot,T)=e^{-\lambda T}\psi_1(\cdot,T)=\tilde r\psi_1(\cdot,0).$$
In view of $\psi_1(\cdot,0)\in P^o$, the Krein-Rutman theorem \cite{KR1950} implies that $\tilde r=r(K)$ and $\tilde u=cu^{*}$ for some constant $c>0$, whence
$\lambda=\lambda_\infty $ and $\psi_1=c\psi$.
The proof is thus complete.
\end{proof}

 \section{Proof of (\ref{ldnaux})}\label{appenC}

\begin{proof}[Proof of {\rm\eqref{ldnaux}}]
Recall from \eqref{definition} that $\lambda^{\mathcal{N}\mathcal{N}}\big((\kappa_0-\delta,\kappa+\delta)\big)$ defines the principle eigenvalue of 
\begin{equation*}
\begin{cases}
\mathcal{L}_0\psi:=\partial_t\psi-\partial_{xx}\psi+V\psi=\lambda\psi, \ \
&x\in (\kappa(t)-\delta,\kappa(t)+\delta), \, t\in (0,T),\\
\partial_x\psi(\kappa-\delta,t)=\partial_x\psi(\kappa+\delta,t)=0, &  t\in \left[0,T\right],\\
\psi(x,0)=\psi(x,T), & x\in (\kappa(0)-\delta,\kappa(0)+\delta).
\end{cases}
\end{equation*}
For any given $\epsilon>0$,
we fix  some $\delta_*=\delta_*(\epsilon)>0$ small such  that
 $|V(x,t)-V\left(\kappa,t\right)|<\epsilon$ for all $x\in\left[\kappa(t)-\delta_*,\kappa(t)+\delta_*\right]$ and $ t\in[0,T]$.
Define
 $$
 \overline{\varphi}(t):=\exp\left[{-\int_0^t V\left(\kappa(s),s\right)\mathrm{d}s+{t\over T}\int_0^T
 V\left(\kappa(s), s\right)\mathrm{d}s}\right].
 $$
 For any $0<\delta<\delta_*$ and $x\in\left(\kappa(t)-\delta,\kappa(t)+\delta\right)$, $t\in[0,T]$, direct calculation gives
 $$\mathcal{L}_0\overline{\varphi}\leq \partial_t\overline{\varphi}-\partial_{xx}\overline{\varphi}+(V\left(\kappa(t),t\right)+\epsilon)\overline{\varphi}=\left[\frac{1}{T}\int_0^TV\left(\kappa(s),s\right)\mathrm{d}s+\epsilon\right]\overline{\varphi}$$
 and
 $$\mathcal{L}_0\overline{\varphi}\geq \partial_t\overline{\varphi}-\partial_{xx}\overline{\varphi}+(V\left(\kappa(t),t\right)-\epsilon)\overline{\varphi}=\left[\frac{1}{T}\int_0^TV\left(\kappa(s),s\right)\mathrm{d}s-\epsilon\right]\overline{\varphi}.$$
Since $\overline{\varphi}$ is $T$-periodic,  we can apply Proposition \ref{appendixprop} to deduce that for any $0<\delta<\delta_*$,
 \begin{equation*}
    \int_0^T V\left(\kappa(s),s\right)\mathrm{d}s-\epsilon\leq \lambda^{\mathcal{N}\mathcal{N}}\big((\kappa-\delta,\kappa+\delta)\big)\leq\frac{1}{T}\int_0^T V\left(\kappa(s),s\right)\mathrm{d}s+\epsilon.
 \end{equation*}
 Letting $\delta\searrow0$ and then $\epsilon\searrow0$  give \eqref{ldnaux}.
 \end{proof}

\bigskip
\noindent{\bf Acknowledgments.} {\small
 We sincerely thank the referees for their valuable suggestions 
which help improve the manuscript.
SL was partially supported by the NSF of China (grant Nos. 1207011419
and 11571364). 
YL  was partially supported by the NSF (grant No. DMS-1853561). RP was partially supported by NSF of China (grant No. 11671175). MZ was partially supported by the Nankai Zhide Foundation
and NSF of China (No. 11971498).}

\end{document}